\documentclass[12pt, twoside]{article}
\usepackage{amsfonts,mathrsfs,amsmath,amssymb,amsthm,upref,amscd}
\usepackage{xcolor}
\usepackage[all,cmtip]{xy}
\usepackage{setspace}
\usepackage{enumerate}
\usepackage[titletoc]{appendix}
\usepackage{times}
\usepackage{cite}
\usepackage{tikz}
\usepackage[colorinlistoftodos]{todonotes}
\makeatletter 
\@mparswitchfalse
\makeatother
\normalmarginpar 
\usetikzlibrary{arrows}
\numberwithin{equation}{section}

\def\titlerunning#1{\gdef\titrun{#1}}
\makeatletter
\def\author#1{\gdef\autrun{\def\and{\unskip, }#1}\gdef\@author{#1}}

\makeatother

\def\subjclass#1{{\renewcommand{\thefootnote}{}%
		\footnote{\emph{Mathematics Subject Classification (2010):} #1}}}
\def\keywords#1{\par\medskip
	\noindent\textbf{Keywords.} #1}

\allowdisplaybreaks

\theoremstyle{plain}
\newtheorem{Thm}{Theorem}[section]
\newtheorem{Lem}[Thm]{Lemma}
\newtheorem{claim}{Claim}
\newtheorem{Cor}[Thm]{Corollary}
\newtheorem{Prop}[Thm]{Proposition}
\newtheorem*{Thm*}{Theorem}
\newtheorem*{claim*}{Claim}
\newtheorem*{Cor*}{Corollary}

\newtheorem*{Ques*}{Question}

\newtheorem*{Prob*}{Problem}
\newtheorem*{OProb*}{Open Problem}
\newtheorem*{SSYProb*}{Singular spinorial Yamabe problem}
\theoremstyle{definition}

\newtheorem*{Def*}{Definition}
\newtheorem{Rem}[Thm]{Remark}
\newtheorem{Ex}{Example}

\DeclareMathOperator{\Spin}{Spin}
\DeclareMathOperator{\SO}{SO}
\DeclareMathOperator{\Aut}{Aut}

\newcommand{\equ}{equation}
\newcommand{\C}{\mathbb{C}}
\newcommand{\N}{\mathbb{N}}
\newcommand{\R}{\mathbb{R}}
\newcommand{\Z}{\mathbb{Z}}

\newcommand\cg{\mathcal{G}}

\newcommand\co{\mathcal{O}}

\def\mbs{\mathbb{S}}

\def\mfs{\mathfrak{S}}
\def\mfd{\mathfrak{D}}
\def\id{\text{Id}}
\def\ig{\textit{g}}

\def\pa {\partial}
\def\op{\oplus}
\def\ot{\otimes}

\def\De{\Delta}

\def\al{\alpha}
\def\bt{\beta}

\def\de{\delta}
\def\Ga{\Gamma}
\def\ga{\gamma}

\def\lm{\lambda}
\def\La{\Lambda}
\def\om{\omega}
\def\Om{\Omega}
\def\sa{\sigma}
\def\vr{\varepsilon}
\def\va{\varphi}

\begin{document}
	\titlerunning{Constructions of Delaunay-type solutions for the spinorial Yamabe equation on spheres}
	
	\title{Constructions of Delaunay-type solutions for the spinorial Yamabe equation on spheres}
	
	\author{Ali Maalaoui \quad Yannick Sire \quad Tian Xu}
	
	\date{}
	
	\maketitle
	
	

	\begin{abstract}
		In this paper we construct singular solutions to the critical Dirac equation on spheres. More precisely, first we construct solutions admitting two points singularities that we call Delaunay-type solutions because of their similarities with the Delaunay solutions constructed for the singular Yamabe problem in \cite{MP1 , Schoen1989}. Then we construct another kind of singular solutions admitting a great circle as a singular set. These solutions are the building blocks for singular solutions on a general Spin manifold.

		\vspace{.5cm}
	\end{abstract}
	
	\tableofcontents
	
\section{Introduction and statement of the main result}

Since the resolution of the Yamabe problem, much has been clarified about the behavior of solutions of the semilinear elliptic equation relating the scalar curvature functions of two
conformally related metrics. One of the starting points for several recent developments was R. Schoen's construction of complete metrics with constant positive scalar curvature on the sphere $S^m$, conformal to the standard round metric, and with prescribed isolated singularities (see \cite{Schoen1988}). In analytical terms, it is equivalent to seeking for a function $u>0$ satisfying
\begin{\equ}\label{Yamabe}
-\De_{\ig_{S^m}}u+\frac{m(m-2)}4u=\frac{m(m-2)}4u^{\frac{m+2}{m-2}} \quad \text{on } S^m\setminus\Sigma, \ m\geq3
\end{\equ}
in the distributional sense with $u$ singular at every point of $\Sigma\subset S^m$. Here we denote by $\ig_{S^m}$ the standard Riemannian metric on $S^m$.

Eq. \eqref{Yamabe} and its counterpart on a general manifold $(M,\ig)$ are known as the {\it singular Yamabe problem}, and has been extensively studied. Just as the classical Yamabe problem in the compact setting, the questions concerning metrics of constant positive scalar curvature are considerably more involved. Remarkable breakthroughs and geometrically appealing examples were obtained by Schoen and Yau \cite{SY1988} and Schoen \cite{Schoen1988} when the ambient manifold is the $m$-sphere $S^m$. The former established that if $S^m\setminus\Sigma$ admits a complete metric with scalar curvature bounded below by a positive constant, then the Hausdorff dimension of $\Sigma$ is at most $(m-2)/2$, and the latter constructed several examples of domains $S^m\setminus\Sigma$ that admit complete conformally flat metrics with constant positive scalar curvature, including the case where $\Sigma$ is any finite set with at least two points. Subsequently, Mazzeo and Smale \cite{MS} and Mazzeo and Pacard \cite{MP1, MP2} generalized the existence results, allowing $\Sigma$ to be a disjoint union of submanifolds with dimensions between $1$ and $(m-2)/2$ when the ambient manifold $(M,\ig)$ is a general compact manifold with constant nonnegative scalar curvature, and between $0$ and $(m-2)/2$ in the case $(M,\ig)=(S^m,\ig_{S^m})$.

In the past two decades, it has been realized that the conformal Laplacian, namely the operator appearing as the linear part of \eqref{Yamabe}, falls into a particular family of operators. These operators are called conformally covariant elliptic operators of order $k$ and of bidegree $((m-k)/2,(m+k)/2)$, acting on manifolds $(M,\ig)$ of dimension $m>k$. Many important geometric operators are in this class, for instance, the conformal Laplacian, the Paneitz operator, the Dirac operator, see also \cite{Branson98, Fegan76, GP} for more examples. All such operators  share several analytical properties, in particular, they are associated to the non-compact embedding of Sobolev space $H^{k/2}\hookrightarrow L^{2m/(m-k)}$. And often, they have a central role in conformal geometry.

Let $(M,\ig,\sa)$ be an $m$-dimensional spin manifold, $m\geq2$, with a fixed Riemannian metric $\ig$ and a fixed spin structure $\sa:P_{\Spin}(M)\to P_{\SO}(M)$. The Dirac operator $D_\ig$ is defined in terms of a representation $\rho:\Spin(m)\to\Aut(\mbs_m)$ of the spin group which is compatible with Clifford multiplication. Let $\mbs(M):=P_{\Spin}(M)\times_\rho\mbs_m$ be the associated bundle, which we call the spinor bundle over $M$. Then the Dirac operator $D_\ig$ acts on smooth sections of $\mbs(M)$, i.e. $D_\ig: C^\infty(M,\mbs(M))\to C^\infty(M,\mbs(M))$, is a first order conformally covariant operator of bidegree $((m-1)/2,(m+1)/2)$. We point out here that the spinor bundle $\mbs(M)$ has complex dimension $2^{[\frac m2]}$.

Analogously to the conformal Laplacian, where the scalar curvature is involved, the Dirac operator on a spin manifold has close relations with the mean curvature function associated to conformal immersions of the universal covering into Euclidean spaces.  This theory is referred as the spinorial Weierstra\ss\ representation, and we refer to \cite{Ammann, Ammann2009, Kenmotsu, Konopelchenko, KS, Friedrich98, Matsutani, Taimanov97, Taimanov98, Taimanov99} and references therein for more details in this direction. In a similar way as in the Yamabe problem, the spinorial analogue of the Yamabe equation (related with a normalized positive constant mean curvature) reads as
\begin{\equ}\label{SY0}
	D_{\ig}\psi=|\psi|_{\ig}^{\frac2{m-1}}\psi \quad \text{on } (M,\ig)
\end{\equ}
where $|\cdot|_\ig$ stands for the induced hermitian metric on fibers of the spinor bundle. One may also consider the equation with an opposite sign
\begin{\equ}\label{SY0-opposite}
	D_{\ig}\psi=-|\psi|_{\ig}^{\frac2{m-1}}\psi \quad \text{on } (M,\ig)
\end{\equ}
which corresponds to negative constant mean curvature surfaces. However, since the spectrum of $D_\ig$ is unbounded on both sides of $\R$ and is symmetric about the origin on many manifolds (say, for instance $\dim M\not\equiv 3 (\text{mod } 4)$), the two problems \eqref{SY0} and \eqref{SY0-opposite} are of the same structure from analytical point of view.

Although conformally covariant operators share many properties, only few statements can be proven simultaneously for all of them. Particularly, the behavior of solutions of the conformally invariant equation \eqref{SY0} or \eqref{SY0-opposite} is still unclear. From the analytic perspective, some of the conformally covariant operators are bounded from below (e.g. the Yamabe and the Paneitz operator), whereas others are not (e.g. the Dirac operator). Some of them act on functions, while others on sections of vector bundles. For the Dirac operators, additional structure (e.g. spin structure) is used for defining it, and hence, more attention needs to be payed on such an exceptional case.

In this paper we initiate an investigation into the singular solutions of the nonlinear Dirac equation \eqref{SY0} when the ambient manifold is $S^m$, which is perhaps the most geometrically appealing instance of this problem. As was described earlier, for a given closed subset $\Sigma\subset S^m$, it is to find metrics $\ig=|\psi|_{\ig_{S^m}}^{4/(m-1)}\ig_{S^m}$ which are complete on $S^m\setminus\Sigma$ and such that $\psi$ satisfies Eq. \eqref{SY0} with $(M,\ig)=(S^m\setminus\Sigma,\ig_{S^m})$. This is the \text{\it singular spinorial Yamabe problem}. Let us mention that, up until now, no existence examples have been known for the singular solutions of Eq. \eqref{SY0}. Our first main result is follows:

\begin{Thm}\label{main thm on sphere}
Let $\Sigma \subset S^m$ be a pair of antipodal points, for $m\geq2$, or an equatorial circle for $m\geq3$. There is a one-parameter family $\mfs_m$ of spinors $\psi$ solving the problem
\begin{\equ}\label{SY0-sphere}
D_{\ig_{S^m}}\psi= |\psi|_{\ig_{S^m}}^{\frac2{m-1}}\psi \quad \text{on } S^m\setminus\Sigma
\end{\equ}
such that $\ig=|\psi|_{\ig_{S^m}}^{\frac4{m-1}}\ig_{S^m}$ is a complete metric on $S^m\setminus\Sigma$. Moreover,
\begin{itemize}
	\item[$(1)$] if $\Sigma$ is a pair of antipodal points, the family $\mfs_m$ is parameterized by $\mu \in [-\frac{(m-1)^m}{2^{m+1}m},+\infty)\setminus\{0\}$.
	
	\item[$(2)$] if $\Sigma$ is an equatorial circle, the family $\mfs_m$ is parameterized by $\co=\cup_{k\in\N}\co_k$, where each $\co_k\subset(0,+\infty)$ is a bounded open set, $\co_k\cap\co_j=\emptyset$ for $k\neq j$ and $\co$ is unbounded.
\end{itemize}
\end{Thm}

\begin{Rem}
Let us remark that Eq. \eqref{SY0-sphere}, or more generally Eq. \eqref{SY0}, is invariant under several Lie group actions. For instance, the canonical action of $S^1=\{e^{i\theta}\in\C:\, \theta\in[0,2\pi]\}$ on spinors keeps the equation invariant (i.e. if $\psi$ is a solution of Eq. \eqref{SY0-sphere} then $e^{i\theta}\psi$ is also a solution, for every fixed $\theta$). Moreover, for the case $m\equiv 2,3,4 (\text{mod } 8)$, the spinor bundle has a quaternionic structure which commutes with Clifford multiplication, see for instance the construction in \cite[Section 1.7]{Friedrich} or \cite[Page 33, Table III]{Lawson}. In these cases, Eq. \eqref{SY0-sphere} is invariant under the action of the unit quaternions $S^3=\{q=\mathbb{H}:\, |q|=1\}$ on spinors. Therefore, in general, it is crucial to distinguish solutions of Dirac equations under various group actions. For instance, these symmetries were exploited in \cite{M1} to construct families of solutions on the sphere and the $S^{1}$ symmetry was used in \cite{M2} to exhibit also non-trivial solutions for the sub-critical Dirac equation. Thanks to our constructions, the solutions in the family $\mfs_m$ obtained in Theorem \ref{main thm on sphere} are distinguished via their parameterizations. And if $\cg$ is a group that keeps Eq. \eqref{SY0-sphere} invariant, our construction shows a larger family $\cg\times\mfs_m$ of singular solutions.
\end{Rem}

As we will see in Section \ref{sec:set up of the problems}, via a conformal change of the metric $\ig_{S^m}$, problem \eqref{SY0-sphere} can be transformed to
\begin{\equ}\label{SY0-Euclidean}
	D_{\ig_{\R^m}}\psi = |\psi|_{\ig_{\R^m}}^{\frac2{m-1}}\psi \quad \text{on } \R^m\setminus\{0\}
\end{\equ}
when $\Sigma$ is a pair of antipodal points
and
\begin{\equ}\label{SY0-Euclidean2}
	D_{\ig_{\R^{m-1}}}\psi = f(x)^{\frac1{m-1}}|\psi|_{\ig_{\R^{m-1}}}^{\frac2{m-1}}\psi \quad \text{on } \R^{m-1}\setminus\{0\}
\end{\equ}
when $\Sigma$ is an equatorial circle, where $f(x)=\frac2{1+|x|^2}$. To obtain the results for Eq. \eqref{SY0-sphere} in consistence with similar results for the classical Yamabe equation, a fundamental idea is to express the equation \eqref{SY0-Euclidean} and \eqref{SY0-Euclidean2} on the cylinder $\R\times S^{l}$, $l=m-1$ or $m-2$. By introducing the cylindrical coordinates $(t,\theta)\in\R\times S^{l}$:
\[
t=-\ln |x|, \quad \theta = \frac{x}{|x|}
\]
for $x\in\R^{l+1}$, one may be expecting that the ansatz
\[
\va(t,\theta)=|x|^{\frac{l}2}\psi(x)
\]
could turn Eq. \eqref{SY0-Euclidean} into a more manageable problem via a separation of variables process leading to a "radial" solution $\psi(x)=\psi(|x|)$. This is the very case for many elliptic problems (with a corresponding change of the exponent on $|x|$), including the Yamabe equation, fractional Yamabe equation \cite{DDGW} and the $Q$-curvature problem \cite{HS}. However, we point out that in the scalar case, there is a symmetrization process that behaves well with elliptic operators, reducing the problem to the study of an ODE. But when dealing with differential operators acting on vector bundles (spinor bundle in our case), one does not have a general symmetrization process. In particular, even on the Euclidean spaces $\R^m$, one cannot use the radial ansatz $\psi=\psi(r)$, $r=|x|$ for $x\in\R^m$, to reduce a Dirac equation to an ODE system in terms of $r$.

Notice that the spinorial Yamabe equation \eqref{SY0-Euclidean} (resp. \eqref{SY0-Euclidean2}) contains $2^{[\frac m2]}$ (resp. $2^{[\frac{m-1}2]}$) unknown complex-functions, which is a considerably large number as $m$ grows. Instead of blindly ``guessing" a particular ansatz, our starting point is the spin structure, or more precisely the spin representation. In fact, we use the matrix representation of Clifford multiplication to construct a ``nice'' function space $E(\R^m)$ for spinor fields which is invariant under the action of the Dirac operator $D_{\ig_{\R^m}}$, see in Section \ref{subsec:ansatz} for the definition. We find that the space $E(\R^m)$ is of particular interest from two perspectives (see Remark \ref{remark-ansatz} below): First of all, when the dimension $m=2,3,4$, $E(\R^m)$ encapsulates several important and special formulations of spinors which are of interest to particle physicists when they study quantum electrodynamic systems. Many important physical simulations have been obtained by using these special spinors,
see for instance \cite{FLR1951, Soler, Wakano, CKSCL}. The second perspective is that, spinors in $E(\R^m)$ reduce the equation \eqref{SY0-Euclidean} significantly in the sense that, for any dimension $m\geq2$, Eq. \eqref{SY0-Euclidean} and \eqref{SY0-Euclidean2} can be reduced to the following ODE systems of only two unknown functions
\begin{\equ}\label{reduced-SY0-1}
\left\{
\aligned
&-f_2'-\frac{m-1}r f_2 = (f_1^2+f_2^2)^{\frac1{m-1}}f_1 \\
&f_1'=(f_1^2+f_2^2)^{\frac1{m-1}}f_2 
\endaligned \quad \text{for } r>0
\right.
\end{\equ}
and
\begin{\equ}\label{reduced-SY0-2}
\left\{
\aligned
&-f_2'-\frac{m-2}r f_2 =\Big(\frac2{1+r^2}\Big)^{\frac1{m-1}} (f_1^2+f_2^2)^{\frac1{m-1}}f_1 \\
&f_1'=\Big(\frac2{1+r^2}\Big)^{\frac1{m-1}}(f_1^2+f_2^2)^{\frac1{m-1}}f_2 
\endaligned \quad \text{for } r>0
\right.
\end{\equ}
where $f_1,f_2\in C^1(0,+\infty)$. After using the Emden-Fowler change of variable $r=e^{-t}$ and writing $f_1(r)=-u(t)e^{\frac{m-1}2t}$, $f_2(r)=v(t)e^{\frac{m-1}2t}$ in \eqref{reduced-SY0-1}, we get a nondissipative Hamiltonian system of $(u,v)$
\begin{\equ}\label{reduced-SY0-1-1}
	\left\{
	\aligned
	&u'+\frac{m-1}2u=(u^2+v^2)^{\frac1{m-1}}v,\\
	-&v'+\frac{m-1}2v=(u^2+v^2)^{\frac1{m-1}}u.
	\endaligned
	\right.
\end{\equ}
And, by writing $f_1(r)=-u(t)e^{\frac{m-2}2t}$ and $f_2(r)=v(t)e^{\frac{m-2}2t}$, we can transform \eqref{reduced-SY0-2} into 
\begin{\equ}\label{reduced-SY0-2-1}
	\left\{
	\aligned
	&u'+\frac{m-2}2u=\cosh(t)^{-\frac1{m-1}}(u^2+v^2)^{\frac1{m-1}}v \\
	-&v'+\frac{m-2}2v=\cosh(t)^{-\frac1{m-1}}(u^2+v^2)^{\frac1{m-1}}u
	\endaligned
	\right.
\end{\equ}
which is a dissipative Hamiltonian system.

Let us denote by
\[
H(u,v)=-\frac{m-1}2uv+\frac{m-1}{2m}(u^2+v^2)^{\frac m{m-1}}
\]
the corresponding Hamiltonian energy for the systems \eqref{reduced-SY0-1-1}. Notice that $H$ is constant along trajectories of \eqref{reduced-SY0-1-1}. Moreover, the equilibrium points of $H$ are 
\begin{\equ}\label{rest points}
(0,0) \quad \text{and} \quad \pm\Big(\frac{(m-1)^{(m-1)/2}}{2^{m/2}},\frac{(m-1)^{(m-1)/2}}{2^{m/2}} \, \Big),
\end{\equ}
where $(0,0)$ is a saddle point and the other two are center  points; then it follows easily that for $\mu\in[-\frac{(m-1)^m}{2^{m+1}m},+\infty)\setminus\{0\}$ there is a periodic solution of \eqref{reduced-SY0-1-1} at the level $\{H=\mu\}$. We set $\mfd_m^1$ for these periodic solutions, parameterized by their Hamiltonian energies. We distinguish a dichotomy within the set $\mfd_{m}^{1}$ based on the sign of the Hamiltonian energy $\mu$. Indeed, $\mfd_{m}^{1}=\mfd_{m}^{1,+}\cup \mfd_{m}^{1,-}$, where 
$$\mfd_{m}^{1,+}:=\{(u,v)\in \mfd_{m}^{1}; H(u,v)>0\} \text{ and } \mfd_{m}^{1,-}:=\{(u,v)\in \mfd_{m}^{1}; H(u,v)<0\}.$$ 
We will call elements of $\mfd_{m}^{1,-}$, positive Delaunay-type solutions and elements of $\mfd_{m}^{1,+}$, sign-changing Delaunay-type solutions for Eq. \eqref{SY0-Euclidean}. This terminology is based on the similarities  between $\mfd_{m}^{1,-}$ and the classical Delaunay solutions for the Yamabe problem. We will clarify more these similarities along the paper. Since any $(u,v)\in\mfd_m^1$ will not reach the rest point $(0,0)$, we have $u(t)^2+v(t)^2$ is bounded away from $0$ for all $t\in\R$. Besides the above existence results, we have the following bifurcation phenomenon for the solutions $(u,v)\in\mfd_m^{1,-}$.

\begin{Thm}\label{main thm ODE1}
Let $m\geq2$, the following facts hold for the system \eqref{reduced-SY0-1-1}:
\begin{itemize}
	\item[$(1)$] For every $T>0$, \eqref{reduced-SY0-1-1} has the constant $2T$-periodic solutions
	\[
	\pm\Big(\frac{(m-1)^{(m-1)/2}}{2^{m/2}},\frac{(m-1)^{(m-1)/2}}{2^{m/2}} \, \Big).
	\]
	Moreover, for $T\leq\frac{\sqrt{m-1}}2\pi$, these are the only solutions to \eqref{reduced-SY0-1-1}.
	
	\item[$(2)$] Let $T>\frac{\sqrt{m-1}}2\pi$ and $d\in\N$ such that $d\frac{\sqrt{m-1}}2\pi<T\leq (d+1)\frac{\sqrt{m-1}}2\pi$. Then \eqref{reduced-SY0-1-1} has $d+1$ inequivalent solutions. Particularly, these solutions are given by the constant solution and $k$ periods of a solution $(u_{T,k},v_{T,k})$ with fundamental period $2T/k$. 
	
	\item[$(3)$] The Hamiltonian energy $H(u_{T,1},v_{T,1})\nearrow 0$ as $T\to+\infty$ and $(u_{T,1},v_{T,1})$ is (locally) compact in the sense that $(u_{T,1},v_{T,1})$ converges in $C_{loc}^1(\R,\R^2)$ to the nontrivial homoclinic solution of \eqref{reduced-SY0-1-1}. That is, there exists $t_{0}\in \R$ such that $(u_{T,1},v_{T,1})$ converges in $C^{1}_{loc}$ to $(u_{0}(\cdot-t_{0}),v_{0}(\cdot-t_{0}))$, where $$u_{0}(t)=\frac{ m^{(m-1)/2}e^{t/2}}{2^{m/2}\cosh(t)^{m/2}} \quad  \text{and} \quad v_0(t)=\frac{ m^{(m-1)/2}e^{-t/2}}{2^{m/2}\cosh(t)^{m/2}}.$$
\end{itemize}
\end{Thm}

By translating the above results to system \eqref{reduced-SY0-1} (hence Eq. \eqref{SY0-Euclidean}), we have
\begin{Cor}
Let $m\geq2$, Eq. \eqref{SY0-Euclidean} has a one-parameter family $\mfs_m^1$ of singular solutions on $\R^m\setminus\{0\}$, parameterized by $[-\frac{(m-1)^m}{2^{m+1}m},+\infty)\setminus\{0\}$. Moreover, the following asymptotic estimates hold
\begin{itemize}
\item $|\psi(x)|\neq 0$, 
\item $|\psi(x)|=O(|x|^{-\frac{m-1}2})$ as $|x|\to+\infty$,
\item $|\psi(x)|=O(|x|^{-\frac{m-1}2})$ as $|x|\to 0$,
\end{itemize} 
for each $\psi\in\mfs_m^1$. Moreover, if $\psi_{\mu}$ is the solution corresponding to $\mu \in [-\frac{(m-1)^m}{2^{m+1}m},0)$, then there exists $\lambda>0$ such that $\psi_{\mu}$ converges in $C^{1}_{loc}(\R^{m})$ to $\psi_{\infty}=\big(\frac{2\lambda}{\lambda^{2}+|x|^{2}}\big)^{\frac{m}{2}}\big(1-\frac{x}{\lambda}\big)\cdot\gamma_{0}$ as $\mu\to 0$, where $\gamma_{0}$ is a constant spinor with $|\gamma_{0}|=\frac{1}{\sqrt{2}}\big(\frac{m}{2}\big)^{\frac{m-1}{2}}$ and ``$\cdot$" stands for the Clifford multiplication on spinors.
\end{Cor}

It is important here to notice the difference between the decay rate of singular solutions that we found in the previous Corollary and the one of regular solutions of \eqref{SY0-Euclidean}, studied in \cite{BorrFr}. Indeed, the decay rate of a regular solution is $O(|x|^{-m+1})$ but the one of a singular solution is $O(|x|^{-\frac{m-1}{2}})$.\\

For the system \eqref{reduced-SY0-2-1} we have
\begin{Thm}\label{main thm ODE2}
Let $m\geq3$, the system \eqref{reduced-SY0-2-1} with initial datum $u(0)=v(0)=\mu>0$ has a solution $(u_\mu,v_\mu)$ globally defined on $\R$. Moreover, there are exactly two types of  initial data, which can be characterized by:
\[
A_k=\Big\{ \mu>0:\, v_\mu \text{ changes sign } k \text{ times on } (0,+\infty) \text{ and } \lim_{|t|\to+\infty}H_\mu(t)<0 \Big\},
\]
and
\[
I_k=\Big\{ \mu>0:\, v_\mu \text{ changes sign } k \text{ times on } (0,+\infty) \text{ and } H_\mu(t)>0 \text{ for all } t\in\R \Big\}
\]
for $k\in\N\cup\{0\}$, where 
\[
H_\mu(t):=-\frac{m-2}2u_\mu v_\mu+\frac{m-1}{2m}\cosh(t)^{-\frac1{m-1}}(u_\mu^2+v_\mu^2)^{\frac m{m-1}}.
\]
In particular,
\begin{itemize}
	\item[$(1)$] $A_k\neq\emptyset$ is a bounded open set for all $k$;
	
	\item[$(2)$] if we set $\mu_k=\sup A_k$, then $\mu_k\in I_k$ and $\mu_0<\mu_1<\dots<\mu_j<\mu_{j+1}<\dots\to+\infty$;
	
	\item[$(3)$] if we set $\nu_k=\sup I_k$, then $\nu_k<+\infty$ and $(\nu_k,\nu_k+\vr)\subset A_{k+1}$ for some small $\vr>0$;
	
	\item[$(4)$] if $\mu\in I_k$, then $(u_\mu(t),v_\mu(t))\to(0,0)$ as $|t|\to\infty$. To be more precise, we have 
	\[
	 u_\mu(t)^2+v_\mu(t)^2=O(e^{-(m-2)t})
	\]
	as $|t|\to+\infty$;
	
	\item[$(5)$] if $\mu\in A_k$, then $u_\mu(t)^2+v_\mu(t)^2$ is bounded from below by a positive constant for all $t\in\R$ and is unbounded as $|t|\to+\infty$; furthermore, up to a multiplication by constant, $u_\mu(t)^2+v_\mu(t)^2$ is upper bounded by $\cosh(t)$ for all $|t|$ large.
\end{itemize}
\end{Thm}

By setting $\mfd_m^2=\{(u_\mu,v_\mu):\, \mu\in\cup_{k\geq0}A_k \}$, we call these unbounded solution the Delaunay-type solution for Eq. \eqref{SY0-Euclidean2}. As a direct consequence of Theorem \ref{main thm ODE2}, we have a characterization of singular solutions for Eq. \eqref{SY0-Euclidean2}
on $\R^{m-1}\setminus\{0\}$.

\begin{Cor}
Let $m\geq3$, Eq. \eqref{SY0-Euclidean} has a one-parameter family $\mfs_m^2$ of singular solutions on $\R^{m-1}\setminus\{0\}$, parameterized by $\cup_{k\geq0}A_k$. Moreover, the following asymptotic estimates hold
\[
|x|^{-\frac{m-2}2}<|\psi(x)|\lesssim |x|^{-\frac{m-1}2} \quad \text{as } |x|\to0
\]
and
\[
|x|^{-\frac{m-2}2}< |\psi(x)|\lesssim |x|^{-\frac{m-3}2} \quad \text{as } |x|\to+\infty
\]
for each $\psi\in\mfs_m^2$
\end{Cor}

This paper is organized as follows. First, in Section 2, we lay down the necessary geometric preliminaries that we will need to formulate our problem, including the main ansatz that will be adopted to find our families of singular solutions. Next, in Section 3, we use the ansatz to formulate the problem as a Hamiltonian system in $\R^{2}$ (autonomous in the case of a point singularity and non-autonomous in the case of a one dimensional singularity). In section 4, we study the properties of the solutions of the Hamiltonian system in the two cases. This allows us to prove Theorems \ref{main thm ODE1} and  \ref{main thm ODE2}. 

\section{Geometric preliminaries}\label{Sec:Geometric preliminaries}

\subsection{General preliminaries about spin geometry}

Let $(M,\ig)$ be an $m$-dimensional Riemannian manifold (not necessarily compact) with a chosen orientation. Let $P_{\SO}(M)$ be the set of positively oriented orthonormal frames on $(M,\ig)$. This is a $\SO(m)$-principal bundle over $M$. A {\it spin structure} on $M$ is a pair $\sa=(P_{\Spin}(M),\vartheta)$ where $P_{\Spin}(M)$ is a $\Spin(m)$-principal bundle over $M$ and $\vartheta: P_{\Spin}(M)\to P_{\SO}(M)$ is a map such that the diagram
\begin{displaymath}
	\xymatrix@R=0.3cm{
		P_{\Spin}(M)\times \Spin(m)  \ar[r] \ar[dd]^{\displaystyle\vartheta\times \Theta}&  P_{\Spin}(M)\ar[dd]^{\displaystyle\vartheta} \ar[dr] \\
		&&  M \\
		P_{\SO}(M)\times \SO(m) \ar[r] & P_{\SO}(M)  \ar[ur] }
\end{displaymath}
commutes, where $\Theta: \Spin(m)\to \SO(m)$ is the nontrivial double covering of $\SO(m)$. There is a topological condition for the existence of a spin structure, namely, the vanishing of the second Stiefel-Whitney class $\om_2(M)\in H^2(M,\Z_2)$. Furthermore, if a spin structure exists, it need not be unique. For these results we refer to \cite{Friedrich, Lawson}.

In order to introduce the spinor bundle, we recall that the Clifford algebra $Cl(\R^m)$ is the associative $\R$-algebra with unit, generated by $\R^m$ satisfying the relation $x\cdot y-y\cdot x=-2(x,y)$ for $x,y\in\R^m$ (here $(\cdot,\cdot)$ is the Euclidean scalar product on $\R^m$). It turns out that $Cl(\R^m)$ has a smallest representation $\rho: \Spin(m)\subset Cl(\R^m)\to End(\mbs_m)$ of dimension $\dim_\C(\mbs_m)=2^{[\frac{m}2]}$ such that $\C l(\R^m):=Cl(\R^m)\otimes\C\cong End_\C(\mbs_m)$ as $\C$-algebra. In case $m$ is 
even, this irreducible representations is uniquely determined, but it splits into non-equivalent sub-representations $\mbs_m^+$ and $\mbs_m^-$ as $\Spin(m)$-representations. If $m$ is odd, there are two irreducible $\C l_m$-representations $\mbs_m^0$ and $\mbs_m^1$. Both of them coincide if considered as $\Spin(m)$-representations. 

Define the chirality operator $\om_\C^{\R^m}=i^{[\frac{m+1}2]}e_1\cdot e_2\cdots e_m\in \C l_m$ with $\{e_1,\dots,e_m\}$ being a positively oriented orthonormal frame on $\R^m$. In case $m$ is even, we have $\om_\C^{\R^m}$ act as $\pm1$ on $\mbs_m^\pm$, and sections of $\mbs_m^+$ (resp. $\mbs_m^-$) are called positive (resp. negative) spinors. While if $m$ is odd, the chirality operator acts on $\mbs_m^j$ as $(-1)^j$, $j=0,1$. Hence, for $m$ odd, it will cause no confusion if we simply identify $\mbs_m^0$ and $\mbs_m^1$ as the same vector space, that is $\mbs_m=\mbs_m^0=\mbs_m^1$, and equip them with Clifford multiplication of opposite sign.

Associated to the above observations, the spinor bundle is then defined as 
\[
\mbs(M):=P_{\Spin}(M)\times_\rho\mbs_m.
\]
Note that the spinor bundle carries a natural Clifford multiplication, a natural hermitian metric and a metric connection induced from the Levi-Civita connection on $TM$ (see \cite{Friedrich, Lawson}), this bundle satisfies the axioms of Dirac bundle in the sense that
\begin{itemize}
	\item[$(i)$] for any $x\in M$, $X,Y\in T_xM$ and $\psi\in\mbs_x(M)$
	\[
	X\cdot Y\cdot \psi + Y\cdot X\cdot\psi + 2\ig(X,Y)\psi = 0;
	\]
	\item[$(ii)$] for any $X\in T_xM$ and $\psi_1,\psi_2\in\mbs_x(M)$,
	\[
	(X\cdot \psi_1,\psi_2)_\ig=-(\psi_1,X\cdot\psi_2)_\ig,
	\]
	where $(\cdot,\cdot)_\ig$ is the hermitian metric on $\mbs(M)$;
	\item[$(iii)$] for any $X,Y\in\Ga(TM)$ and $\psi\in\Ga(\mbs(M))$,
	\[
	\nabla_X^\mbs(Y\cdot\psi)=(\nabla_XY)\cdot\psi+Y\cdot\nabla_X^\mbs\psi,
	\]
	where $\nabla^\mbs$ is the metric connection on $\mbs(M)$.
\end{itemize}
The Dirac operator is then defined on the spinor bundle $\mbs(M)$ as the composition
\[
\begin{array}{ccccccc}
	D_\ig:\Ga(\mbs(M)) & \stackrel{\nabla^\mbs}{\longrightarrow} & \Ga(T^*M\otimes \mbs(M)) & \longrightarrow & \Ga(TM\otimes \mbs(M)) & \stackrel{\mathfrak{m}}{\longrightarrow} & \Ga(\mbs(M))
\end{array}
\]
where $\mathfrak{m}$ denotes the Clifford multiplication $\mathfrak{m}: X\otimes\psi\mapsto X\cdot\psi$.

Let us remark that there is an implicit $\ig$-dependence in the Clifford multiplication ``$\mathfrak{m}$" or ``$\cdot$". In fact, considering a simple case where we replace $\ig$ with a conformal metric $\tilde\ig=e^{2u}\ig$, the isometry $X\mapsto e^{-u}X$ from $(TM,\ig)$ onto $(TM,\tilde\ig)$ defines a principal bundle isomorphism $\SO(TM,\ig)\to\SO(TM,\tilde\ig)$ lifting to the spin level. Then it induces a bundle isomorphism $\mbs(M,\ig)\to\mbs(M,\tilde\ig)$, $\psi\mapsto\tilde\psi$, fiberwisely preserving the Hermitian inner product and sending $X\cdot\psi$ to $e^{-u}X\tilde\cdot\tilde\psi$. In the sequel, when necessary, we shall write $D_\ig^M$ and $\cdot_\ig$, etc., to precise the underlying manifold $M$ and the metric $\ig$.

\subsection{Spinor bundle and the Dirac operator on product manifolds}

In this subsection our notation is close to \cite{SX2020}. Let $(N=M_1\times M_2,\ig_N=\ig_{M_1}\op\ig_{M_2})$ be a product of Riemannian spin $m_j$-manifolds $(M_j,\ig_{M_j},\sa_{M_j})$, $j=1,2$. We have
\[
P_{\Spin}(N)=(P_{\Spin}(M_1)\times P_{\Spin}(M_2))\times_\zeta\mbs_{m_1+m_2}
\]
where $\zeta:\Spin(m_1)\times \Spin(m_2)\to\Spin(m_1+m_2)$ is the Lie group homomorphism lifting the standard embedding $\SO(m_1)\times \SO(m_2)\to\SO(m_1+m_2)$.

The spinor bundle over $N$ can be identified with
\[
\mbs(N)=\left\{
\aligned
&(\mbs(M_1)\op\mbs(M_1))\ot \mbs(M_2) &\quad & \text{both } m_1 \text{ and } m_2 \text{ are odd}, \\
&\qquad \mbs(M_1)\ot\mbs(M_2) &\quad & m_1 \text{ is even}.
\endaligned \right.
\]
That is, we always put the even dimensional factor in the place of $M_1$. And the Clifford multiplication on $\mbs(N)$ can be explicitly given in terms of the Clifford multiplications on its factors. In fact,
for $X\in TM_1$, $Y\in TM_2$, $\va\in\Ga(\mbs(M_2))$ and
\[
\psi=\begin{cases}
\psi_1\op\psi_2\in \Ga(\mbs(M_1)\op\mbs(M_1)) & \text{for both
	$m_1$ and $m_2$ odd} \\
\qquad  \psi\in\Ga(\mbs(M_1)) & \text{for $m_1$ even}
\end{cases}
\]
we have
\begin{\equ}\label{product spinor representation}
	(X\op Y)\cdot_{\ig_N} (\psi\ot\va)=(X\cdot_{\ig_{M_1}}\psi)\ot\va+(\om_\C^{M_1}\cdot_{\ig_{M_1}}\psi)\ot(Y\cdot_{\ig_{M_2}}\va)
\end{\equ}
where in case $m_1$ and $m_2$ odd we set $X\cdot_{\ig_{M_1}}\psi=(X\cdot_{\ig_{M_1}}\psi_1)\op(-X\cdot_{\ig_{M_1}}\psi_2)$ and $\om_\C^{M_1}\cdot_{\ig_{M_1}}\psi=i(\psi_2\op-\psi_1)$. Let us remark that there are different ways to formulate the Clifford multiplication \eqref{product spinor representation}, but such changes are equivalent. Indeed, due to the uniqueness of $\C l(TM_1\op TM_2)$, any definition of the Clifford multiplication on $\mbs(N)$ can be identified with \eqref{product spinor representation} via a vector bundle isomorphism (see the examples in the next subsection).

Let $\nabla^{\mbs(M_1)}$ and $\nabla^{\mbs(M_2)}$ be the
Levi-Civita connections on $\mbs(M_1)$ and $\mbs(M_2)$. By
\[
\nabla^{\mbs(M_1)\ot\mbs(M_2)}=\nabla^{\mbs(M_1)}\ot\id_{\mbs(M_2)}
+\id_{\mbs(M_1)}\ot \nabla^{\mbs(M_2)}
\]
we mean the tensor product connection on $\mbs(M_1)\ot\mbs(M_2)$. Then, by \eqref{product spinor representation}, the Dirac operator on $N$ is given by
\begin{\equ}\label{product-Dirac-operator}
D_{\ig}^N = \tilde D_{\ig_{M_1}}^{M_1} \ot \id_{\mbs(M_2)}+ (\om_\C^{M_1}\cdot_{\ig_{M_1}}\id_{\mbs(M_1)})\ot D_{\ig_{M_2}}^{M_2}
\end{\equ}
where
$\tilde D_{\ig_{M_1}}^{M_1}= D_{\ig_{M_1}}^{M_1}\op -D_{\ig_{M_1}}^{M_1}$ if both $m_1$ and $m_2$ are odd
and $\tilde D_{\ig_{M_1}}^{M_1}=D_{\ig_{M_1}}^{M_1}$ if $m_1$ is even.

For the case $m_1+m_2$ even, we have the decomposition
$\mbs(N)=\mbs(N)^+\op\mbs(N)^-$ and, moreover, when restrict $D_{\ig}^N$
on those half-spinor spaces we get
$D_{\ig}^N:\Ga(\mbs(N)^\pm)\to\Ga(\mbs(N)^\mp)$.

\subsection{A particular ansatz in Euclidean spaces}\label{subsec:ansatz}

Let $M=\R^m$ be equipped with the Euclidean metric, then the spinor bundle is given by $\mbs(\R^m)=\R^m\times\mbs_m\cong\R^m\times\C^{2^{[\frac m2]}}$. Although, from the abstract setting, the Dirac operator can be given by
\[
D_{\ig_{\R^m}}\psi=\sum_{k=1}^m e_k\cdot_{\ig_{\R^m}}\nabla_{e_k}\psi, \quad \psi\in\mbs(\R^m)
\]
where $\{e_1,\dots,e_m\}$ is a orthonormal base of $\R^m$, we can have a more explicit representation of this operator. In fact the Dirac operator can be formulated as a constant coefficient differential operator of the form
\begin{\equ}\label{Dirac-operator-in-matrices}
D_{\ig_{\R^m}}=\sum_{k=1}^{m}\al_k^{(m)}\frac{\pa}{\pa x_k}
\end{\equ}
where $\al_k^{(m)}$ is a linear map $\al_k^{(m)}:\C^{2^{[\frac m2]}}\to\C^{2^{[\frac m2]}}$ satisfying the relation 
\begin{\equ}\label{Dirac-relation}
\al_j^{(m)}\al_k^{(m)}+\al_k^{(m)}\al_j^{(m)}=-2\de_{ij}
\end{\equ}
for all $j,k$.

Let us give a possible construction of these $\{\al_j^{(m)}\}$ by using $2^{[\frac m2]}\times 2^{[\frac m2]}$ complex matrices with a block structure. We start with $m=1$ and the $1$-dimensional Dirac operator $D_{\ig_\R}=i\frac{d}{dx}$, that is we have $\al_1^{(1)}=i$ the pure imaginary unit. For $m$ is even, we define
\[
\al_j^{(m)}=\begin{pmatrix}
	\textbf{0} & -i \al_j^{(m-1)} \\[0.3em]
	i\al_j^{(m-1)} & \textbf{0}
\end{pmatrix}
\quad \text{for } j=1,\dots,m-1 \quad \text{and}\quad 
\al_m^{(m)}=\begin{pmatrix}
	\textbf{0} & i\,\id  \\[0.3em]
	i\,\id  & \textbf{0}
\end{pmatrix}
\]
where ``$\id$" is understood to be the identity on $\C^{2^{[\frac{m-1}2]}}$. And, if $m$ is odd, we define
\[
\al_j^{(m)}=\al_j^{(m-1)} \quad \text{for } j=1,\dots,m-1 \quad \text{and} \quad
\al_m^{(m)}=i^{\frac{m+1}2}\al_1^{(m-1)}\cdots\al_{m-1}^{(m-1)}.
\]

It is illuminating to consider this construction in low dimensions:
\begin{Ex}\label{Example m=2}
For $m=2$, we have
\[
\al_1^{(2)}=\begin{pmatrix}
	0 & 1 \\
	-1 & 0
\end{pmatrix} \quad \text{and} \quad 
\al_2^{(2)}=\begin{pmatrix}
	0 & i\\
	i & 0 
\end{pmatrix}.
\]
Writing a spinor field $\psi:\R^2\to\mbs(\R^2)$ in components as $\begin{pmatrix}
	\psi_1\\
	\psi_2
\end{pmatrix}\in\C^2$, we then have
\begin{\equ}\label{2D-Dirac-operator}
D_{\ig_{\R^2}}\psi=\begin{pmatrix}
	0 & 1 \\
	-1 & 0
\end{pmatrix}
\begin{pmatrix}
\frac{\pa\psi_1}{\pa x_1}\\[0.3em]
\frac{\pa \psi_2}{\pa x_1}
\end{pmatrix} + \begin{pmatrix}
0 & i\\
i & 0 
\end{pmatrix}
\begin{pmatrix}
	\frac{\pa\psi_1}{\pa x_2}\\[0.3em]
	\frac{\pa \psi_2}{\pa x_2}
\end{pmatrix}
= \begin{pmatrix}
	\frac{\pa\psi_2}{\pa x_1}+i\frac{\pa \psi_2}{\pa x_2} \\[0.3em]
	-\frac{\pa\psi_1}{\pa x_1}+i\frac{\pa \psi_1}{\pa x_2}
\end{pmatrix}.
\end{\equ}
Thus, in this case, the Dirac operator is simply the Cauchy-Riemann operator. 

Consider the product $\R^2=\R\times\R$ and the identification $\mbs(\R^2)=(\mbs(\R)\op\mbs(\R))\otimes\mbs(\R)$. We see that the fiberwise isomorphism is given explicitly by  
\begin{\equ}\label{2D-fiberwise-isomorphism}
(\mbs(\R)\op\mbs(\R))\otimes\mbs(\R)\ni \begin{pmatrix}
u_1v\\
u_2v
\end{pmatrix} \longleftrightarrow 
\frac1{\sqrt2}\begin{pmatrix}
	(u_1+u_2)v \\
	(u_1-u_2)v
\end{pmatrix} \in \mbs(\R^2)
\end{\equ}
for $u_1,u_2,v\in\Ga(\mbs(\R))$. In particular, by \eqref{product-Dirac-operator}, we see that
\[
\begin{pmatrix}
i\frac{d}{dx} & 0 \\
0 & -i\frac{d}{dx}
\end{pmatrix}\begin{pmatrix}
u_1v\\
u_2v
\end{pmatrix}- \frac{d}{dy}\begin{pmatrix}
u_2v\\
-u_1v
\end{pmatrix}=\begin{pmatrix}
iu_1'v-u_2v'\\
-iu_2'v+u_1v'
\end{pmatrix}
\]
which coincides with \eqref{2D-Dirac-operator} (under the action of the isomorphism in \eqref{2D-fiberwise-isomorphism}).
\end{Ex}

\begin{Ex}
For $m=3$, we have
\[
\al_1^{(3)}=\begin{pmatrix}
	0 & 1 \\
	-1 & 0
\end{pmatrix}, \quad
\al_2^{(3)}=\begin{pmatrix}
	0 & i\\
	i & 0 
\end{pmatrix} \quad \text{and} \quad 
\al_3^{(3)}=\begin{pmatrix}
	-i & 0\\
	0 & i 
\end{pmatrix}
\]
which are exactly the classical Pauli matrices. And for the product $\R^3=\R^2\times\R$, it is easy to obtain from \eqref{Dirac-operator-in-matrices} that
\[
D_{\ig_{\R^3}}=D_{\ig_{\R^2}}\otimes \id_{\mbs(\R)}+\begin{pmatrix}
	-1 & 0\\
	0 & 1
\end{pmatrix}\otimes D_{\ig_\R}
\]
fitting into \eqref{product-Dirac-operator}.
\end{Ex}

\begin{Ex}
	For $m=4$, we have
	\[
	\al_1^{(4)}=\begin{pmatrix}
		\begin{matrix}
			\text{\LARGE 0}
		\end{matrix} & \begin{matrix}
			& -i \,\\
			i& 
		\end{matrix} \\
		\begin{matrix}
			& i \\
			-i& 
		\end{matrix} &\begin{matrix}
			\text{\LARGE 0}
		\end{matrix}
	\end{pmatrix}, \quad 
\al_2^{(4)}=\begin{pmatrix}
	\begin{matrix}
		\text{\LARGE 0}
	\end{matrix} & \begin{matrix}
		& 1 \,\\
		1& 
	\end{matrix} \\
	\begin{matrix}
		& -1 \\
		-1& 
	\end{matrix} &\begin{matrix}
		\text{\LARGE 0}
	\end{matrix}
\end{pmatrix}, \quad
\al_3^{(4)}=\begin{pmatrix}
	\begin{matrix}
		\text{\LARGE 0}
	\end{matrix} & \begin{matrix}
		-1&0  \,\\
		0&1 
	\end{matrix} \\
	\begin{matrix}
		1&0  \\
		0&-1 
	\end{matrix} &\begin{matrix}
		\text{\LARGE 0}
	\end{matrix}
\end{pmatrix}
	\]
and
\[
\al_4^{(4)}=\begin{pmatrix}
	\begin{matrix}
		\text{\LARGE 0}
	\end{matrix} & \begin{matrix}
		i&0  \\
		0&i \
	\end{matrix} \\
	\begin{matrix}
     \, i&0   \\
		0&i \
	\end{matrix} &\begin{matrix}
		\text{\LARGE 0}
	\end{matrix}
\end{pmatrix}
\]
And for the product $\R^4=\R^2\times\R^2$, we have $\mbs(\R^4)=\mbs(\R^2)\otimes\mbs(\R^2)$. By considering a bundle isomorphism
\[
\mbs(\R^2)\otimes\mbs(\R^2)\ni 
\begin{pmatrix}
	u_1\\
	u_2
\end{pmatrix}\otimes \begin{pmatrix}
v_1\\
v_2
\end{pmatrix} \longleftrightarrow 
\begin{pmatrix}
	-iu_1v_1\\
	-iu_2v_2\\
	iu_1v_2\\
	iu_2v_1
\end{pmatrix}\in\mbs(\R^4)
\]
for $u_1,u_2,v_1,v_2\in\Ga(\mbs(\R^2))$, one easily verifies the correspondence 
\[
D_{\ig_{\R^4}}= D_{\ig_{\R^2}}\otimes \id_{\mbs(\R^2)}+ \begin{pmatrix}
	-1 &0 \\
	0&1
\end{pmatrix}\otimes D_{\ig_{\R^2}}
\]
which justifies \eqref{product-Dirac-operator}. Meanwhile, for the product $\R^4=\R^3\times\R$ and the associated spinor bundle $\mbs(\R^4)=(\mbs(\R^3)\op\mbs(\R^3))\otimes\mbs(\R)$,  we have the fiberwise isomorphism
\[
(\mbs(\R^3)\op\mbs(\R^3))\otimes\mbs(\R)\ni 
\begin{pmatrix}
	\psi_1\va\\
	\psi_2\va\\
	\psi_3\va\\
	\psi_4\va
\end{pmatrix}\longleftrightarrow
\frac1{\sqrt2}\begin{pmatrix}
	(\psi_4-\psi_2)\va\\
	(\psi_3-\psi_1)\va\\
	(\psi_2+\psi_4)\va\\
	-(\psi_1+\psi_3)\va
\end{pmatrix} \in\mbs(\R^4)
\] 
for $\begin{pmatrix}
	\psi_1\\
	\psi_2
\end{pmatrix}, \begin{pmatrix}
\psi_3\\
\psi_4
\end{pmatrix}\in\mbs(\R^3)$ and $\va\in\mbs(\R)$
such that the action of
\[
\begin{pmatrix}
	D_{\ig_{\R^3}} & 0\\
	0 &- D_{\ig_{\R^3}}
\end{pmatrix} \otimes \id_{\mbs(\R)}+ i\begin{pmatrix}
0 & \id_{\mbs(\R^3)}\\
-\id_{\mbs(\R^3)}&0
\end{pmatrix} \otimes D_{\ig_\R}
\]
on $(\mbs(\R^3)\op\mbs(\R^3))\otimes\mbs(\R)$ coincides with the action of $D_{\ig_{\R^4}}$ on $\mbs(\R^4)$. This verifies \eqref{product-Dirac-operator}.
Note the analogy with dimension two.
\end{Ex}

We could continue this analysis. For general $m$, one can compute the matrices $\{\al_j^{(m)}\}$, the chirality operator $\om_\C^{\R^m}$ and, particularly when $m$ is even, the corresponding bundle isomorphism to decompose the Dirac operator in a product structure. However, these explicit formulas are seldom. It is always simpler to use the abstract setting of the Clifford module.

It is interesting to note that the aforementioned explicit formula for the Dirac operator motivates a ``nice" function space which is invariant under the actions of the Dirac operator. More precisely, let us set
\[
\aligned
E(\R^m)&:=\Big\{ \psi(x)=f_1(|x|)\ga_0+\frac{f_2(|x|)}{|x|}x\cdot\ga_0  \,:\, x\in\R^m, \
f_1,f_2\in C^\infty(0,\infty) \text{ and } \ga_0\in S_\C^{2^{[\frac m2]}}\Big\} \\
&=\Big\{ \psi(x)=f_1(|x|)\ga_0+\frac{f_2(|x|)}{|x|}\sum_{k=1}^m x_k\al_k^{(m)}\ga_0  \,:\,
f_1,f_2\in C^\infty(0,\infty) \text{ and } \ga_0\in S_\C^{2^{[\frac m2]}}\Big\}.
\endaligned
\] 
where $S_\C^{2^{[\frac m2]}}$ stands for the complex unit sphere in the spin-module $\mbs_m\cong\C^{2^{[\frac m2]}}$.
Then, following the rule of the Clifford multiplication or the relation \eqref{Dirac-relation}, it is easy to check that
\[
D_{\ig_{\R^m}}\psi=-\Big(f_2'(|x|)+\frac{(m-1)f_2(|x|)}{|x|}\Big)\ga_0+\frac{f_1'(|x|)}{|x|}x\cdot\ga_0\in E(\R^m) \quad \forall \psi\in E(\R^m).
\]
Moreover, in order to make sure that $\psi$ is continuous at the origin, one may consider a further restriction to the subspace
\[
E_0(\R^m)=\Big\{ \psi(x)=f_1(|x|)\ga_0+\frac{f_2(|x|)}{|x|}x\cdot\ga_0 \in E \,:\, f_1'(t)=O(t) \text{ and } f_2(t)=O(t) \text{ as } t\searrow0\Big\}.
\]

\begin{Rem}\label{remark-ansatz}
\begin{itemize}
	\item[(1)] It is interesting to see that the specific ansatz provided in $E(\R^m)$ contains some important formulations of spinors, which are of interest to many physicists when they are dealing with spinor fields in quantum electrodynamics. In fact, to the best of our knowledge, it can be traced back to R. Finkelstein, R. LeLevier and M. Ruderman \cite{FLR1951} in 1951 when they investigated a nonlinear Dirac equation in $\R^3\times\R$. By separating the time variable, the authors introduced a very special formulation of a spinor field, i.e.
	\begin{\equ}\label{soler1}
		\psi(r,\theta_1,\theta_2)=\begin{pmatrix}
			f_1(r)\\[0.3em]
			0\\[0.3em]
			if_2(r)\cos\theta_1 \\[0.3em]
			if_2(r)\sin\theta_1 e^{i\theta_2}
		\end{pmatrix} \text{ or }
	\begin{pmatrix}
		if_2(r)\cos\theta_1 \\[0.3em]
		if_2(r)\sin\theta_1 e^{i\theta_2}\\[0.3em]
		f_1(r)\\[0.3em]
		0
	\end{pmatrix}
	\end{\equ}
where $(r,\theta_1,\theta_2)\in(0,+\infty)\times[0,\pi]\times[0,2\pi]$ is the spherical coordinates on $\R^3$. And subsequently, this ansatz has been commonly used in particle physics where spinors play a crucial role, see for instance \cite{Soler, Wakano} and \cite{CKSCL} for a $2$-dimensional analogue. Now, in our setting, we understand that the above spinor field belongs to the sub-bundle $\mbs(\R^3)\op\mbs(\R^3)$. Consider the standard spherical coordinates 
\[
x_1=r\cos\theta_1, \quad x_2=r\sin\theta_1\cos\theta_2, \quad
x_3=r\sin\theta_1\sin\theta_2\cos\theta_3
\]
and
\[
x_4=r\sin\theta_1\sin\theta_2\sin\theta_3
\]
for $r>0$, $\theta_1,\theta_2\in[0,\pi]$ and $\theta_3\in[0,2\pi]$, if we restrict to $\theta_2=\frac{\pi}2$ (i.e. the variable $x_2$ is separated out, treated as the time variable) and take 
\[
\ga_0=\begin{pmatrix}
	1\\
	0\\
	0\\
	0
\end{pmatrix}\in S_\C^4,
\]
we soon derive that
\[
f_1(|x|)\ga_0+\frac{f_2(|x|)}{|x|}\sum_{k=1}^4 x_k\al_k^{(4)}\ga_0=\begin{pmatrix}
	if_2(r)\cos\theta_1 \\[0.3em]
	if_2(r)\sin\theta_1 e^{i\theta_3}\\[0.3em]
	f_1(r)\\[0.3em]
	0
\end{pmatrix}
\]
which is exactly the latter one in \eqref{soler1}. 

\item[(2)] Although the special ansatz \eqref{soler1} for a spinor has been known for over half a century, it is still new and important to have the family $E(\R^m)$ for general dimensions. Particularly, the ansatz in $E(\R^m)$ reduces the Dirac equation significantly. Indeed, for the semilinear equations of the form
\begin{\equ}\label{model form}
D_{\ig_{\R^m}}\psi=h(|x|,|\psi|)\psi, \quad 
\psi:\R^m\to \mbs_m\cong \C^{2^{[\frac m2]}}
\end{\equ}
where $h:[0,+\infty)\times [0,\infty)\to\R$ is a given function, 
the ansatz in $E(\R^m)$ transforms it equivalently to
\[
\left\{
\aligned
&-f_2'-\frac{m-1}{r}f_2=h\Big(r,\sqrt{f_1^2+f_2^2} \, \Big)f_1 ,\\
&f_1'=h\Big(r,\sqrt{f_1^2+f_2^2} \,\Big)f_2,
\endaligned\right. \quad \text{for } r>0
\]
making the problem much easier to deal with.

\item[(3)] This ansatz was also used to study several mathematical physics models. We cite for instance \cite{Borr1, Borr2, BorrFr} for the study of Dirac-type equation, \cite{FSY, Rot} for the study of particle like solutions of coupled Dirac type equations.

\item[(4)] The space $E(\R^m)$ is somehow natural within spinor fields. Indeed, if one looks at the parallel spinors on $\R^{m}$ and the Dirac bubbles \cite{BMW} (corresponding to Killing spinors on the sphere), then one notices that they all belong to $E(\R^{m})$. Hence, we can think about $E(\R^{m})$ as a generalized special class of spinors.
\end{itemize}

\end{Rem}

\section{Set up of the problems}\label{sec:set up of the problems}

Let us consider the $m$-sphere $S^m$ to be $\R^m\cup\{\infty\}$, where the coordinates $x\in\R^m$ is given by the standard stereographic projection from the north pole $\al_m:S^m\setminus\{P_N^{m+1}\}\to\R^m$ (here $P_N^{m+1}=(0,\dots,0,1)\in S^m\subset\R^{m+1}$ stands for the north pole). For clarity, we use the sub- or superscripts to indicate the underlying dimensions. By setting $P_S^{m+1}=(0,\dots,0,-1)$ for the south pole, we can see that the manifold $\R\times S^{m-1}$ is conformally equivalent to $S^m\setminus\{P_N^{m+1},\,P_S^{m+1}\}$. The conformal diffeomorphism can be explicitly formulated by
\begin{\equ}\label{conformal-equivalence1}
\begin{array}{ccccc}
S^m\setminus\{P_N^{m+1},\, P_S^{m+1}\} & \stackrel{\al_m}{\longrightarrow} & \R^m\setminus\{0\} & \stackrel{\bt_m}{\longrightarrow} & \R\times S^{m-1} \\[0.6em]
\xi=(\xi_1,\dots,\xi_{m+1}) & \longmapsto & x=(x_1,\dots,x_m) & \longmapsto & (\ln|x|,x/|x|)
\end{array}
\end{\equ}
where we have $(\al_m^{-1})^*\ig_{S^m}=\frac4{(1+|x|^2)^2}\ig_{\R^m}$ and  $(\bt_m)^*(\ig_{\R}\op\ig_{S^{m-1}})=\frac1{|x|^2}\ig_{\R^m}$.

\medskip

This observation leads to some further considerations. Typical examples arise from the (connected) domain $\Omega\subset S^n$ whose complement is an equatorial circle. Without loss of generality, we may consider the domain
\[
S^m\setminus S^1=\Big\{(\xi_1,\dots,\xi_{m+1})\in\R^{m+1}:\, \sum_{k}\xi_k^2=1,\ \xi_1^2+\xi_{m+1}^2<1\Big\}.
\]
Then we have the following conformal equivalence
\begin{\equ}\label{conformal-equivalence2}
\begin{array}{ccccc}
	\Om=S^m\setminus S^1 & \stackrel{\al_m}{\longrightarrow} &  \R^m\setminus\{(\R,0,\dots,0)\} & \stackrel{\bt_m}{\longrightarrow} & \R\times (S^{m-1}\setminus \{P_N^m,P_S^m\})
\end{array}
\end{\equ}

We now consider the solutions of the spinorial Yamabe equation on the sphere $(S^m,\ig_{S^m})$, that are singular at a prescribed closed set $\Sigma\subset S^m$. More specifically, we will consider the problem
\begin{\equ}\label{SY-sphere}
	D_{\ig_{S^m}}\phi=|\phi|_{\ig_{S^m}}^{\frac2{m-1}}\phi \quad \text{on } \Om=S^m\setminus\Sigma
\end{\equ}
 when $\Sigma$ is given by a pair of antipodal points, say $\{P_N^{m+1},\, P_S^{m+1}\}$, or an equatorial circle $S^1$.
 
Before discussing the Delaunay family of solutions to Eq. \eqref{SY-sphere}, let us recall the transformation formula of the Dirac operator under conformal changes (see \cite{Hij86, Hit74}):
\begin{Prop}\label{conformal formula}
	Let $\ig_0$ and $\ig=f^2\ig_0$ be two conformal metrics on
	a Riemannian spin $m$-manifold $M$. Then, there exists an
	isomorphism of vector bundles $F:\, \mbs(M,\ig_0)\to
	\mbs(M,\ig)$ which is a fiberwise isometry such that
	\[
	D_\ig\big( F(\psi) \big)=F\big( f^{-\frac{m+1}2}D_{\ig_0}
	\big( f^{\frac{m-1}2}\psi \big)\big),
	\]
	where $D_{\ig_0}$ and $D_{\ig}$ are the Dirac operators on $M$ with respect to the metrics $\ig_0$ and $\ig$, respectively.
\end{Prop}
\noindent
In what follows, our discussions will be build upon this formula.

\subsection{The singular set is a pair of antipodal points}

In this setting, without loss of generality, we assume $\Sigma=\{P_N^{m+1},\, P_S^{m+1}\}\subset S^m$. Then,  as a direct consequence of Proposition \ref{conformal formula}, we have that if $\psi$ is a solution to the equation
\begin{\equ}\label{singular-Dirac1}
D_{\ig_{\R^m}}\psi=|\psi|_{\ig_{\R^m}}^{\frac2{m-1}}\psi \quad \text{on } \R^m\setminus\{0\}
\end{\equ}
then $\phi=F(f^{-\frac{m-1}2}\psi)$ ($f(x)=\frac2{1+|x|^2}$) is a solution to Eq. \eqref{SY-sphere}. Notice that since Eq. \eqref{singular-Dirac1} has the same structure  as \eqref{model form}, we shall look at solutions of the form
\begin{\equ}\label{the ansatz}
\psi(x)=f_1(|x|)\ga_0+\frac{f_2(|x|)}{|x|}x\cdot\ga_0\in E(\R^m).
\end{\equ}
Then, applying the Emden-Fowler change of variable $r=e^{-t}$ and write $f_1(r)=-u(t)e^{\frac{m-1}2t}$ and $f_2(r)=v(t)e^{\frac{m-1}2t}$, we are led to consider the following system
\begin{\equ}\label{reduced-SY1}
	\left\{
	\aligned
	&u'+\frac{m-1}2u=(u^2+v^2)^{\frac1{m-1}}v,\\
	-&v'+\frac{m-1}2v=(u^2+v^2)^{\frac1{m-1}}u.
	\endaligned
	\right.
\end{\equ}
This system is easily integrated and is nondissipative, in particular, the Hamiltonian energy
\[
H(u,v)=-\frac{m-1}2uv+\frac{m-1}{2m}\big(u^2+v^2\big)^{\frac{m}{m-1}}
\]
is constant along solutions of \eqref{reduced-SY1}. 

The equilibrium points for system \eqref{reduced-SY1} are
\[
(0,0) \quad \text{and} \quad \pm\Big(\frac{(m-1)^{(m-1)/2}}{2^{m/2}},\frac{(m-1)^{(m-1)/2}}{2^{m/2}} \, \Big).
\]
And there is a special homoclinic orbit
\begin{\equ}\label{homoclinic solution}
u_0(t)=\frac{ m^{(m-1)/2}e^{t/2}}{2^{m/2}\cosh(t)^{m/2}}, \quad
v_0(t)=\frac{ m^{(m-1)/2}e^{-t/2}}{2^{m/2}\cosh(t)^{m/2}}
\end{\equ}
corresponding to the level set $H=0$; it limits on the origin as $t$ tends to $\pm\infty$, and encloses a bounded set $\La$ in the first quadrant of the $(u,v)$-plane, given by $\{H\leq 0\}$. It is easy to see that orbits not enclosed by this level set, i.e. those orbits in $\{H>0\}$, must pass across the $u$-axis and $v$-axis. That is $u$ and $v$ must change sign. Observe that the equilibrium point $(0,0)$ is contained exactly in two orbits: the homoclinic one and the stationary orbit $(0,0)$. Hence, for orbits $(u(t),v(t))$ in $\{H\neq0\}$, we must have that $u^2+v^2\neq0$ for all $t$. And thus, we have an unbounded one parameter family of periodic solutions
\[
\mfd_m^1 = \bigg\{(u,v) \text{ is a solution to Eq. \eqref{reduced-SY1}}: \, u(0)=v(0)=\mu>0,\ \mu\neq\frac{m^{(m-1)/2}}{2^{m/2}} \,\bigg\},
\]
which induces correspondingly a family of singular solutions $\mfs_m^1$ to Eq. \eqref{singular-Dirac1} via \eqref{the ansatz}. Remark that $|\psi(x)|\to+\infty$ as $|x|\to0$ and $|\psi(x)|=O(|x|^{-\frac{m-1}2})$ as $|x|\to+\infty$ for each $\psi\in\mfs_m^1$. Therefore, these solutions give rise to distinguished singular solutions of Eq. \eqref{SY-sphere}.

If we take into account just the periodic solutions in $\mfd_m^1$, we will call them  the Delaunay-type solutions of the spinorial Yamabe problem \eqref{singular-Dirac1}. Although we do not know them explicitly, in Section \ref{sec:analysis of the ODEs}, we will study the bifurcation phenomenon for solution in the first quadrant of $(u,v)$-plane.

\subsection{The singular set is an equatorial circle}

First of all, we need to observe that Eq. \eqref{SY-sphere} can be interpreted as an equation on $\R\times(S^{m-1}\setminus\{P_N^m,P_S^m\})$ by a conformal change of the Riemannian metric $\ig_{S^m}$ on $S^m\setminus S^1$. Consider the product metric on $\R\times S^{m-1}$, given in $(\tau,\vartheta)$-coordinates by $\bar\ig=d\tau^2+d\vartheta^2$, where $\vartheta=(\vartheta_1,\dots,\vartheta_{m-1})$ parameterizes the unit sphere $S^{m-1}$. Then it follows from the conformal equivalence \eqref{conformal-equivalence2} that
\[
(\al_m^{-1}\circ\bt_m^{-1})^*\ig_{S^m}=\frac{4e^{2\tau}}{(1+e^{2\tau})^2}\bar\ig=\frac1{\cosh(\tau)^2}\bar\ig.
\]  
And as a direct consequence of Proposition \ref{conformal formula}, we have that if $\varphi$ is a solution to the equation
\begin{\equ}\label{SY-n-odd1}
	D_{\bar\ig}\va=|\va|_{\bar\ig}^{\frac2{m-1}}\va \quad \text{on } \R\times(S^{m-1}\setminus\{P_N^m,P_S^m\})
\end{\equ}
then $\phi=F(\cosh(\tau)^{\frac{m-1}2}\va)$ is a solution to Eq. \eqref{SY-sphere} with $F$ being a bundle isomorphism.

Let us remark that the formula \eqref{product-Dirac-operator} on product manifolds indicates a way to construct singular solutions for Eq. \eqref{SY-n-odd1}. In fact, if $m$ is odd (hence $m\geq3$), then $m-1$ is even and we can consider a special spinor of the form $\va=1\otimes\tilde\psi$ so that Eq. \eqref{SY-n-odd1} is reduced to
\begin{\equ}\label{reduced-SY2}
	D_{\ig_{S^{m-1}}}\tilde\psi=|\tilde\psi|_{\ig_{S^{m-1}}}^{\frac2{m-1}}\tilde\psi
\end{\equ}
where $\tilde\psi=\tilde\psi(\vartheta)$ is a spinor on $S^{m-1}\setminus\{P_N^m,P_S^m\}$. And once again, by using the conformal formula in Proposition \ref{conformal formula}, Eq. \eqref{reduced-SY2} can be equivalently transformed to
\begin{\equ}\label{reduced-SY3}
	D_{\ig_{\R^{m-1}}}\psi=f(x)^{\frac1{m-1}}|\psi|_{\ig_{\R^{m-1}}}^{\frac2{m-1}}\psi \quad \text{on } \R^{m-1}\setminus\{0\}
\end{\equ}
where $f(x)=\frac2{1+|x|^2}$ for $x\in\R^{m-1}$. And the solutions of \eqref{reduced-SY2} and \eqref{reduced-SY3} are in one-to-one correspondence via the identification $\tilde\psi \leftrightarrow f^{-\frac{m-2}{2}}\psi$ for spinors.

Now, by considering the ansatz
\[
	\psi(x)=f_1(|x|)\ga_0+\frac{f_2(|x|)}{|x|}x\cdot\ga_0\in E(\R^{m-1}).
\]
and applying the change of variable $r=e^{-t}$, we can reduce Eq. \eqref{reduced-SY3} to the system
\begin{\equ}\label{reduced-SY4}
	\left\{
	\aligned
	&u'+\frac{m-2}2u=\cosh(t)^{-\frac1{m-1}}(u^2+v^2)^{\frac1{m-1}}v \\
	-&v'+\frac{m-2}2v=\cosh(t)^{-\frac1{m-1}}(u^2+v^2)^{\frac1{m-1}}u
	\endaligned
	\right.
\end{\equ}
where $f_1(r)=-u(t)e^{\frac{m-2}2t}$ and $f_2(r)=v(t)e^{\frac{m-2}2t}$. 

\medskip

If $m$ is even, then the spinor bundle on $\R\times(S^{m-1}\setminus\{P_N^m,P_S^m\})$ can be identified with $\mbs(\R)\otimes(\mbs(S^{m-1})\op\mbs(S^{m-1}))$ and the Dirac operator can be formulated as
\[
D_{\bar\ig}=\begin{pmatrix}
	D_{\ig_{S^{m-1}}} & 0\\
	0 & -D_{\ig_{S^{m-1}}}
\end{pmatrix} \otimes \id_{\mbs(\R)} + i\begin{pmatrix}
0 & \id_{\mbs(S^{m-1})} \\
-\id_{\mbs(S^{m-1})}& 0
\end{pmatrix} \otimes D_{\ig_\R}.
\]
Hence, considering a spinor of the form $\va=1\otimes(\tilde\psi_1\op\tilde\psi_1)$ for $\tilde\psi_1,\tilde\psi_1\in\Ga(\mbs(S^{m-1}))$, we may reduce Eq. \eqref{SY-n-odd1} to the following Dirac system
\[
\begin{pmatrix}
	D_{\ig_{S^{m-1}}}\tilde\psi_1\\[0.3em]
	-D_{\ig_{S^{m-1}}}\tilde\psi_2
\end{pmatrix}
=\big( |\tilde\psi_1|_{\ig_{S^{m-1}}}^2+|\tilde\psi_2|_{\ig_{S^{m-1}}}^2 \big)^{\frac1{m-1}}\begin{pmatrix}
	\tilde\psi_1\\[0.3em]
	\tilde\psi_2
\end{pmatrix}
\]
on $S^{m-1}\setminus\{P_N^m,P_S^m\}$. Similar to Eq. \eqref{reduced-SY3}, we can transform the above system to
\begin{\equ}\label{reduced-SY5}
	\begin{pmatrix}
		D_{\ig_{\R^{m-1}}}\psi_1\\[0.3em]
		-D_{\ig_{\R^{m-1}}}\psi_2
	\end{pmatrix}
	=f(x)^{\frac1{m-1}}\big( |\psi_1|_{\ig_{\R^{m-1}}}^2+|\psi_2|_{\ig_{\R^{m-1}}}^2 \big)^{\frac1{m-1}}\begin{pmatrix}
		\psi_1\\[0.3em]
		\psi_2
	\end{pmatrix}
\end{\equ}
on $\R^{m-1}\setminus\{0\}$.

Now, using the ansatz
\[
\psi_1(x)=f_1(|x|)\ga_0+\frac{f_2(|x|)}{|x|}x\cdot\ga_0 \quad \text{and} \quad
\psi_2(x)=f_3(|x|)\ga_0+\frac{f_4(|x|)}{|x|}x\cdot\ga_0 
\]
in $E(\R^{m-1})$ and applying the change of variable $r=e^{-t}$, we then get the following system
\begin{\equ}\label{reduced-SY6}
	\left\{
	\aligned
	&u_1'+\frac{m-2}2u_1=\cosh(t)^{-\frac1{m-1}}\big( u_1^2+u_2^2+v_1^2+v_2^2 \big)^{\frac1{m-1}}v_1 \\
   -&v_1'+\frac{m-2}2v_1=\cosh(t)^{-\frac1{m-1}}\big( u_1^2+u_2^2+v_1^2+v_2^2  \big)^{\frac1{m-1}}u_1 \\
   &u_2'+\frac{m-2}2u_2=\cosh(t)^{-\frac1{m-1}}\big( u_1^2+u_2^2+v_1^2+v_2^2  \big)^{\frac1{m-1}}v_2 \\
   -&v_2'+\frac{m-2}2v_2=\cosh(t)^{-\frac1{m-1}}\big( u_1^2+u_2^2+v_1^2+v_2^2  \big)^{\frac1{m-1}}u_2
	\endaligned
	\right.
\end{\equ}
where we have substituted  $f_1(r)=-u_1(t)e^{\frac{m-2}2t}$, $f_2(r)=v_1(t)e^{\frac{m-2}2t}$, $f_3(r)=u_2(t)e^{\frac{m-2}2t}$ and $f_4(r)=v_2(t)e^{\frac{m-2}2t}$. Therefore, we can consider the solutions for which $u_1=u_2$ and $v_1=v_2$; these are the solutions having the simplest and clearest structure. By writing $u=\sqrt2 u_1$ and $v=\sqrt2v_1$, we can turn \eqref{reduced-SY6} into
\[
		\left\{
	\aligned
	&u'+\frac{m-2}2u=\cosh(t)^{-\frac1{m-1}}\big( u^2+v^2 \big)^{\frac1{m-1}}v \\
	-&v'+\frac{m-2}2v=\cosh(t)^{-\frac1{m-1}}\big( u^2+v^2  \big)^{\frac1{m-1}}u 
	\endaligned
	\right.
\]
which exactly coincides with  \eqref{reduced-SY4}.

Clearly, the system \eqref{reduced-SY4} has an Hamiltonian structure, where the Hamiltonian energy is given by
\[
H(t,u,v)=-\frac{m-2}{2}uv+\frac{m-1}{2m}\cosh(t)^{-\frac1{m-1}}(u^2+v^2)^{\frac{m}{m-1}}.
\]
It is evident that this system is dissipative and there is no periodic solution. However, one may consider solutions that are not converging to $(0,0)$ as $t\to\pm\infty$. More precisely, we will characterize the following family of solutions
\[
\mfd_m^2=\big\{ (u,v) \text{ is a solution to Eq. \eqref{reduced-SY4}}:\, u^2(t)+v^2(t)\to+\infty \text{ as } t\to\pm\infty \big\}
\]
which induces a family of singular solutions $\mfs_m^2$ to Eq.~\eqref{reduced-SY2}. Hence these solutions gives rise to singular solutions of Eq.~\eqref{SY-sphere}. In this setting, we shall call the family $\mfd_m^2$ the Delaunay-type solutions.

\section{Analysis of the ODE systems}\label{sec:analysis of the ODEs}

This section contains our main study of the dynamical systems \eqref{reduced-SY1} and \eqref{reduced-SY4}. We point out that both systems have a variational structure. In fact, if we denote  $z=(u,v)\in\R^2$,  systems \eqref{reduced-SY1} and \eqref{reduced-SY4} can be  rewritten as
\begin{\equ}\label{Hamiltionian system}
\dot{z}=\frac{dz}{dt}=J\nabla_z H(t,z)
\end{\equ} 
where 
\[
J=\begin{pmatrix}
	0&1\\
	-1&0
\end{pmatrix}
\]
and $H$ stands for the corresponding Hamiltonian energy. The functionals
\[
\Phi_T(z)=\frac12\int_{-T}^T(-J\dot{z},z)dt-\int_{-T}^{T}H(t,z)dt
\]
and 
\[
\Phi(z)=\frac12\int_\R(-J\dot{z},z)dt-\int_{\R}H(t,z)dt
\]
can be used to obtain periodic solutions and homoclinic solutions for \eqref{Hamiltionian system} respectively. In particular, there is one-to-one correspondence between  $2T$-periodic solutions of \eqref{Hamiltionian system} and critical points of $\Phi_T$ (as long as $H(t,z)$ is periodic in the $t$-variable or independent of $t$). Similarly, critical points of $\Phi$ correspond to homoclinic solutions of \eqref{Hamiltionian system}, i.e., $z(t)\to(0,0)$ as $t\to\pm\infty$. 

For the autonomous system, i.e. \eqref{reduced-SY1}, we point out that the existence of a $2T$-periodic solution for every $T>T_0$, some $T_0>0$, and the asymptotic behavior of these solutions as $T\nearrow+\infty$ have been already investigated in \cite{Tanaka, AM}. By summarizing their results, we have
\begin{Prop}\label{known results for the ODE1}
	There exists $T_0>0$ such that for every $T>T_0$ the Hamiltonian system \eqref{reduced-SY1} has a non-constant $2T$-periodic solution $z_T$. The family $\{z_T:\, T>T_0\}$ is compact in the following sense: for any sequence $T_n\nearrow+\infty$, up to a subsequence if necessary, $z_{T_n}$ converges in $C_{loc}^1(\R,\R^2)$ to a nontrivial solution $z_\infty$ of the system \eqref{reduced-SY1} on $\R$ satisfying
	\[
	\lim_{|t|\to+\infty}z_\infty(t)=\lim_{|t|\to+\infty}\dot{z}_\infty(t)=0,
	\]
	i.e., $z_\infty$ is a homoclinic orbit. 
\end{Prop}

Notice that the previous proposition does not provide a clear description of the behavior of the solutions $z_{T}$ as $T\searrow T_{0}$ or a characterization of $z_{\infty}$. For instance, from the arguments in \cite{Tanaka, AM}, we do not have an estimate of $T_{0}$ and we do not know if there are non-constant solutions below $T_{0}$. In fact, if $H$ has a ``good" structure around its equilibrium points, then one can use Lyapunov’s center theorem to exhibit a family of small amplitude periodic solutions bifurcating from the equilibrium solution and also have an estimate on $T_{0}$. Nevertheless, this does not provide uniqueness of the family of non-constant solutions.\\

In the sequel, we will perform different approaches to characterize the Delaunay-type families $\mfd_m^1$ and $\mfd_m^2$. We also want to point out that an alternative method can be used to find periodic solutions of family $\mfd_m^{1,-}$ using variational analysis and by tracking the least energy solution, we can characterize the homoclinic energy $z_{\infty}$, corresponding to the least energy solution for the functional $\Phi$. This procedure was used in a more general setting of product manifolds in \cite{BX}. 

\subsection{The nondissipative case: Bifurcation of the positive periodic orbits}

In order to analyse the dynamical system \eqref{reduced-SY1}, we recall that
\[
H(u,v)=-\frac{m-1}2uv+\frac{m-1}{2m}\big(u^2+v^2\big)^{\frac{m}{m-1}}
\]
for $u,v\in\R$ and $m\geq2$, which is independent of $t$. We will focus on the periodic solutions/orbits of \eqref{reduced-SY1} in the first quadrant of the $(u,v)$-plane, that is $u,v:\R/2T\Z\to (0,+\infty)$ for all $T>0$. Such solutions will be referred as positive solutions.

System \eqref{reduced-SY1} has an ``obvious" constant solution $u=v\equiv\frac{(m-1)^{(m-1)/2}}{2^{m/2}}$ for all $T>0$. From now on, we intend to look at non-constant solutions. By setting $z=u^2+v^2$ and $w=u^2-v^2$, we have $uv=\frac{\sqrt{z^2-w^2}}2$ and  \eqref{reduced-SY1} becomes
\begin{\equ}\label{ODE1}
	\left\{
	\aligned
	& z'=-2\lm w \\
	& zz'-ww' = \frac1\lm z^{p-1}z'\sqrt{z^2-w^2}
	\endaligned
	\right.
\end{\equ}
where we denote $\lm=\frac{m-1}2>0$ and $p=\frac{m}{m-1}\in(1,2]$ for simplicity. After multiplication by $(z^2-w^2)^{-1/2}$ in the second equation, we obtain
\[
\frac{d}{dt}\big(\sqrt{z^2-w^2} \,\big)=\frac{d}{dt}\Big( \frac1{\lm p}z^p \Big).
\]
Thus, for any solution $z$ and $w$, there exists a constant $K$ such that $\sqrt{z^2-w^2}=\frac1{\lm p}z^p+K$, that is,
\begin{\equ}\label{ODE1-1}
	w^2=z^2-\Big(\frac1{\lm p}z^p +K \Big)^2 \quad \text{and} \quad \frac1{\lm p}z^p +K\geq0.
\end{\equ}

For $K\in\R$, let us denote
\[
F_K(s)=s^2-\Big(\frac1{\lm p}s^p +K \Big)^2 \quad \text{for } s\geq0.
\]
Remark that, if $(z,w)$ is a non-constant $2T$-periodic solution of \eqref{ODE1}, then $z$ must achieve the maximum and minimum in one period. Hence $z'$ has at least two zeros. This, together with the first equation in \eqref{ODE1}, implies that $F_K$ should vanish at least twice. Therefore, the conditions on $K$ are particularly restrictive. In fact, for $K=0$, we can combine the first equation in \eqref{ODE1} and \eqref{ODE1-1} together to obtain $(z')^2=4\lm^2 z^2-\frac4{p^2}z^{2p}$. Then, if there exist $t_0$ and $t_1$ such that $z(t_0)<z(t_1)$ and $z'(t_0)=z'(t_1)=0$, we have $z(t_0)=0$ and $z(t_1)=(\frac m2)^{m-1}$. Clearly, this should corresponds to the homoclinic solution \eqref{homoclinic solution} and can not be periodic. For $K<0$, by analyzing the algebraic equation $F_K(s)=0$, we can see that $F_k$ has exactly two zeros $0<s_0<s_1$ on $(0,+\infty)$ given by the relations
\[
\left\{
\aligned
&s_0=-\frac1{\lm p}s_0^p -K, \\
&s_1=\frac1{\lm p}s_1^p+K.
\endaligned
\right.
\]
But we find $\frac1{\lm p}s_0^p+K<0$, which fails to satisfy the second inequality in \eqref{ODE1-1}. So the remaining range for $K$ is $(0,+\infty)$. However, it is obvious that $K$ can not be large. 

\begin{Lem}\label{ODE1-lemma1}
If $K>0$ is small, $F_K$ has exactly two zeros on $(0,+\infty)$.
\end{Lem}
\begin{proof}
We only prove the case $p=\frac{m}{m-1}\in(1,2)$, i.e. $m>2$, since $p=2$ is much easier.
Notice that
\[
F_K'(s)=2s-\frac2\lm\Big( \frac1{\lm p}s^p+K \Big)s^{p-1}
\]
for $s\geq0$ and $p\in(1,2]$, we have $F_K'(0)=0$ and $F_K'(s)<0$ in $(0,\de_1)$ for some $\de_1>0$ small. 

Observe that the two maps $s\mapsto\lm s^{2-p}$ and $s\mapsto\frac1{\lm p}s^p+K$ have exactly two intersections for $K>0$ small enough. We denote the horizontal coordinates of these two intersections by $0<s_{0,1}<s_{0,2}$. Then we have $F_K'<0$ on $(0,s_{0,1})\cup(s_{0,2},+\infty)$ and $F_K'>0$ on $(s_{0,1},s_{0,2})$. Therefore,  $F_K(s_{0,1})<0$ is a strict local minimum, whereas $F_K(s_{0,2})$ is a strict local maximum. 

Since $F_0(1)=1-\frac1{\lm^2p^2}>0$ (we used the facts $\lm=\frac{m-1}2$, $p=\frac{m}{m-1}$ and $m>2$), we have $F_K(1)>0$ for all small $K$. Hence $F_K(s_{0,2})>0$. This implies $F_K$ has exactly two zeros on $(0,+\infty)$.
\end{proof}

Let
\[
K_0:=\sup\big\{K>0\,:\, F_K \text{ has two zeros}\big\}.
\]
We remark that, for $K>0$, $F_K$ can not have a third zero in $(0,+\infty)$ since $F_K'$ changes sign at most twice and $F_K(0)<0$.

\begin{Lem}\label{ODE1-lemma2}
$K_0<+\infty$ and $F_{K_0}$ has only one zero, which is the global maximum. Furthermore, $F_K(s)<0$ for all $K>K_0$ and $s\geq0$.
\end{Lem}
\begin{proof}
	Since $K_0<+\infty$ is obvious, we only need to check the remaining statements. To begin with, we mention that 
	\begin{\equ}\label{FK1}
	\frac{\pa}{\pa K} F_K(s)=-2\Big(\frac1{\lm p} s^p +K \Big)<0
	\end{\equ}
	provided that $K>0$ and $s\geq0$. Hence, if $F_{\hat K}(s_{\hat K})>0$ for some $\hat K>0$ and $s_{\hat K}>0$, we have $F_K(s_{\hat K})>0$ for all $K\in(0,\hat K]$. Moreover, due to the continuity of $F_K$ with respect to $K$, there exists $\vr>0$ such that $F_K(s_{\hat K})>0$ for $K\in(\hat K,\hat K+\vr)$. Therefore, we can see that $\big\{K>0\,:\, F_K \text{ has two zeros}\big\}=(0,K_0)$ is an open interval and that $\max F_{K_0}\leq0$ (otherwise $F_{K_0}$ will have two zeros).	By choosing a sequence $K_n\nearrow K_0$ and $s_n>0$ such that $F_{K_n}(s_n)>0$, we have $\{s_n\}$ is bounded and $F_{K_n}(s_n)\to0$ as $n\to\infty$. Therefore $F_{K_0}$ has only one zero, which is the global maximum. The last assertion comes from the fact \eqref{FK1}.
\end{proof}

\begin{Rem}\label{rem:K0}
	The value of $K_0$ can be explicitly computed. Precisely, we have
	\[
	K_0=\Big( 1-\frac1p \Big)\lm^{\frac1{p-1}}=\frac1m\Big(\frac{m-1}2\Big)^{m-1}.
	\]
	In fact, $K=K_0$ is the largest positive number such that the equation $s= \frac1{\lm p}s^p+K$ has a solution.
\end{Rem}

In the sequel, let $K\in(0,K_0)$, we set $0<s_0<s_1$ the points such that $F_K$ vanishes. It is worth pointing out that $s_0$ and $s_1$ are functions of $K$. Then $F_K$ is positive on the interval $(s_0,s_1)$. And Eq. \eqref{ODE1-1} is now equivalent to 
\[
\frac{dz}{2\lm\sqrt{F_K(z)}}=\pm dt,
\]
which can be solved by $\eta_K(z)=\pm t+C$, where
\[
\eta_K(z)=\int_{s_0}^s \frac{dz}{2\lm\sqrt{F_K(z)}}
\]
and $C\in\R$ is a constant.

Of course, $\eta_K$ is defined on the interval $(s_0,s_1)$. By noting that $s_0$ and $s_1$ are simple roots of $F_K$ (that is $F_K'(s_j)\neq0$ for $j=0,1$), we have $\eta_K(s_1)$ is well-defined. Moreover, we have $\eta_K'(s)>0$ and $\eta_K'(s)\to+\infty$ as $s\to s_0$ or $s_1$. Therefore, $\eta_K$ has an inverse $\eta_K^{-1}$ which increases from $s_0$ to $s_1$ on the interval $[0,\eta_K(s_1)]$. Now, solutions to \eqref{ODE1-1} can be represented as $z(t)=\eta_K^{-1}(\pm t+C)$ for $C\in\R$.

Setting
\begin{\equ}\label{zK}
z_K(t)=\begin{cases}
\eta_K^{-1}(t) & t\in[0,\eta_K(s_1)],\\
\eta_K^{-1}(-t) & t\in[-\eta_K(s_1),0],
\end{cases}
\end{\equ}
it follows that $z_K$ is a $2\eta_K(s_1)$-periodic solution of Eq. \eqref{ODE1} and can not have smaller period. Moreover, this $z_K$ (jointly with the corresponding $w_K$ from Eq. \eqref{ODE1}) gives rise to a positive solution $(u_K,v_k)$ of Eq. \eqref{reduced-SY1} with $H(u_K,v_K)=-\frac{\lm K}2<0$.

\begin{Lem}\label{ODE1-lemma3}
The mapping $K\mapsto\eta_K(s_1)$ is continuous. Particularly,
\[
\lim_{K\searrow0}\eta_K(s_1)=+\infty \quad \text{and} \quad \lim_{K\nearrow K_0}\eta_K(s_1)=\frac{\sqrt{m-1}}2\pi
\]
\end{Lem}
\begin{proof}
For starters, we shall write $s_0=s_0(K)$ and $s_1=s_1(K)$ to emphasize that $s_0$ and $s_1$ are functions of $K$. Notice that $s_0$ and $s_1$ are solutions to the equation $s=\frac1{\lm p}s^p+K$. By the implicit function theorem, we have $s_0$ and $s_1$ are $C^1$ functions, in particular,
\[
\left\{
\aligned
&\Big( 1-\frac1\lm s_0(K)^{p-1} \Big)s_0'(K)=1, \\
&\Big( 1-\frac1\lm s_1(K)^{p-1} \Big)s_1'(K)=1.
\endaligned
\right.
\]
Since we have assumed $s_0<s_1$, we have
\[
\Big( 1-\frac1\lm s_0(K)^{p-1} \Big)>0 \quad \text{and} \quad \Big( 1-\frac1\lm s_1(K)^{p-1} \Big)<0
\]
which implies that $s_0'(K)>0$ and $s_1'(K)<0$.

\medskip

The continuity of $\eta_K(s_1)$ is obvious and, without digging out very much from the function $\eta_K(s_1)$, we can evaluate the asymptotic behavior of $\eta_K(s_1)$ as $K$ goes to the end points $0$ and $K_0$. In fact, to see the limiting behavior of $\eta_K(s_1)$ as $K\searrow0$, we first observe that $F_K(0)<0$ and $F_K(2K)>0$ for all small $K$. Hence we have $0<s_0(K)<2K$. Moreover $\lm^{1/(p-1)}<s_1(K)$ since $s_1(K)$ is the larger solution to the equation $s= \frac1{\lm p}s^p+K$. Then
\[
\eta_K(s_1)\geq \int_{2K}^{\lm^{1/(p-1)}}\frac{dz}{2\lm\sqrt{F_K(z)}}\geq \frac1{2\lm} \int_{2K}^{\lm^{1/(p-1)}}\frac{dz}{z}=\frac1{2\lm}\big(\ln\lm^{1/(p-1)} -\ln2K \big).
\]
Thus, by taking $K\to0$, we have $\lim_{K\searrow0}\eta_K(s_1)=+\infty$.

For $K$ close to $K_0$, we set $G_K(t)=F_K(t^{m-1})$, that is
\[
G_K(t)=t^{2(m-1)}-\Big(\frac2m t^m +K \Big)^2.
\]
By writing $t_0=s_0^{1/(m-1)}$ and $t_1=s_1^{1/(m-1)}$, we can write $G_K$ in its factorization
\[
G_K(t)=\frac4{m^2}(t-t_0)(t_1-t) P_K(t)
\]
with 
\[
P_K(t)=\Big( t^m+\frac m2 t^{m-1} +\frac m2 K \Big) \big( a_0t^{m-2}+a_1t^{m-3} +\cdots+ a_{m-3}t + a_{m-2} \big),
\]
where
\[
a_0=1,\quad a_1=t_0+t_1-\frac m2
\]
and 
\[
a_j=-t_0t_1 a_{j-2}+(t_0+t_1)a_{j-1} \quad \text{for } j=2,\dots,m-2.
\]
From elementary computations, we can simply write
\begin{\equ}\label{aj}
a_j=\frac{t_1^{j+1}-t_0^{j+1}}{t_1-t_0}-\frac m2\frac{t_1^{j}-t_0^{j}}{t_1-t_0}
\end{\equ}
for $j=0,1,\dots,m-2$. Then we can reformulate $\eta_K(s_1)$ as
\begin{\equ}\label{etaK-2}
\eta_K(s_1)=\int_{t_0}^{t_1}\frac{t^{m-2}dt}{\sqrt{G_K(t)}}=\frac m2\int_0^1 \frac{(t_0+(t_1-t_0)\tau)^{m-1}d\tau}{\sqrt{\tau(1-\tau)P_K(t_0+(t_1-t_0)\tau)}}.
\end{\equ}

Notice that, as $K$ approaches $K_0$, we have $t_0,t_1\to\frac{m-1}2$. By the continuity of $\eta_K(s_1)$, we have
\[
\lim_{K\to K_0}\eta_K(s_1)=c_m\int_0^1\frac{d\tau}{\sqrt{\tau(1-\tau)}}=c_m\pi
\]
where 
\[
c_m=\frac{m(m-1)^{m-1}}{2^m\sqrt{P_{K_0}(\frac{m-1}2)}}=\frac{\sqrt{m-1}}2.
\]
This completes the proof.
\end{proof}

\begin{Rem}
Recall that we are looking at the $2\eta_K(s_1)$-periodic solutions of Eq. \eqref{ODE1}, then Lemma \ref{ODE1-lemma3} implies:
\begin{itemize}
	\item[(1)] For every $T>0$, Eq. \eqref{ODE1} has the constant solution $z_0\equiv\frac{(m-1)^{m-1}}{2^{m-1}}$ and $w_0\equiv0$, which gives the nontrivial constant solution of Eq. \eqref{reduced-SY1}. And, for $T\leq\frac{\sqrt{m-1}}2\pi$, this is the only possible solution of Eq. \eqref{ODE1}.
	
	\item[(2)] Let $d\in\N$ with $d\frac{\sqrt{m-1}}2\pi<T\leq (d+1)\frac{\sqrt{m-1}}2\pi$. Then for any $k=1,\dots,d$, we have $\frac{T}k\geq \frac Td>\frac{\sqrt{m-1}}2\pi$ and there exists $K=K(T/k)\in(0,K_0)$ such that $\eta_K(s_1)=T/k$.
	
	\item[(3)] The solutions given by \eqref{zK} corresponds to the solutions obtained in Proposition \ref{known results for the ODE1}, since the Hamiltonian energy $H(u_K,v_K)\to0$ and the minimal period  $\eta_K(s_1)\to+\infty$ as $K\to0$. Moreover, we have $T_0=\frac{\sqrt{m-1}}2\pi$.
\end{itemize}
\end{Rem}

We end this section by comparing the classical Delaunay solutions that appear in the study of the singular Yamabe problem and the solutions that we have just studied above. 
Let us recall the classical Delaunay solutions for the singular Yamabe problem as in \cite{MP1,Schoen1989}, that are obtained by solving the ODE 
\begin{\equ}\label{ODE-Yamabe}
u''-\frac{(m-2)^{2}}{4}u+\frac{m(m-2)}{4}u^{\frac{m+2}{m-2}}=0, \quad u>0.
\end{\equ}
This equation is clearly nondissipative, and the corresponding Hamiltonian energy is 
\[
\widetilde{H}(u,u')=\frac12|u'|^{2}-\frac{(m-2)^{2}}{8}u^{2}+\frac{(m-2)^{2}}{8}u^{\frac{2m}{m-2}}.
\]
By examining the level sets of $\widetilde H$, we see that all bounded positive solutions of Eq. \eqref{ODE-Yamabe} lie in the region of the $(u,u')$-plane where $\widetilde H$ is non-positive.
In the figures below, we show a few orbits for both the Hamitonians for the systems \eqref{reduced-SY1} and \eqref{ODE-Yamabe} when $m=3$.

\begin{figure}[h!]
	\begin{minipage}[b]{0.45\textwidth}
		\centering
		\scalebox{0.9}[0.9]
		{ \includegraphics{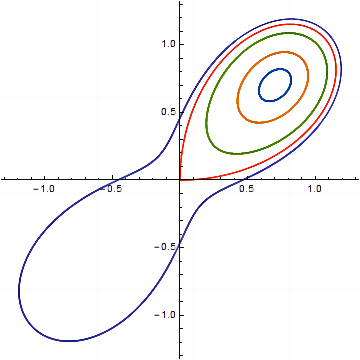}}
		\caption{The orbits for the spinorial Yamabe equation}
	\end{minipage}
	\hfill
	\begin{minipage}[b]{0.45\textwidth}
		\centering
		\scalebox{0.9}[0.9]
		{ \includegraphics{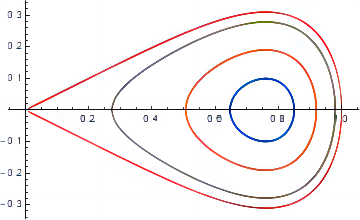}}
		\vspace{3em}
		\caption{The orbits for the classical Yamabe equation}
	\end{minipage}
\end{figure}

\subsection{The dissipative case: Shooting method}

In this subsection, we investigate the system \eqref{reduced-SY4}. In particular, since we are looking for singular solutions of the spinorial Yamabe equation, we are interested in solutions of \eqref{reduced-SY4} such that
\[
(u(t),v(t))\not\to (0,0) \quad \text{as } t\to\pm\infty.
\]

In order to avoid unnecessary complexity and to get non-trivial solutions, we choose as initial conditions
\[
u(0)=v(0)=\mu\in\R\setminus\{0\}.
\]
Moreover, the symmetry of the system allows us to consider only the case $\mu>0$.

Recall that the Hamiltonian energy associated to \eqref{reduced-SY4} is given by
\[
H(t,u,v)=-\frac{m-2}2uv+\frac{m-1}{2m}\cosh(t)^{-\frac1{m-1}}(u^2+v^2)^{\frac m{m-1}}.
\]
We begin with:
\begin{Lem}\label{ODE2-lemma1}
	For any $\mu>0$, there is $(u_\mu,v_\mu)\in C^1(\R,\R^2)$, unique solution of \eqref{reduced-SY4} satisfying $u_\mu(0)=v_\mu(0)=\mu$. Furthermore, $(u_\mu,v_\mu)$ depends continuously on $\mu$, uniformly on $[-T,T]$, for any $T>0$.
\end{Lem}
\begin{proof}
	To begin with, we may write the system \eqref{reduced-SY4} in integral form as
	\[
	\left\{
	\aligned
	&u(t)=\mu + \int_0^t \Big[ \cosh(s)^{-\frac1{m-1}}\big(u(s)^2+v(s)^2\big)^{\frac1{m-1}}v(s) - \frac{m-2}2u(s) \Big] ds \\
	&v(t)=\mu - \int_0^t \Big[ \cosh(s)^{-\frac1{m-1}}\big(u(s)^2+v(s)^2\big)^{\frac1{m-1}}u(s) - \frac{m-2}2v(s) \Big] ds
	\endaligned
	\right.
	\]
for $t\geq0$. Since the right-hand side of the above equation is a Lipschitz continuous function of $(u,v)$, the classical contraction mapping argument gives us a local existence of $(u_\mu,v_\mu)$ on $[0,\de)$. Let $[0,T_\mu)$ be the maximal interval of existence for $(u_\mu,v_\mu)$.

Clearly, if we define $u_\mu(t):=v_\mu(-t)$ and $v_\mu(t):=u_\mu(-t)$ for $t<0$, we have $(u_\mu,v_\mu)$ is a solution on $(-T_\mu,T_\mu)$. Suppose that $T_\mu<+\infty$. Then we have $|u_\mu(t)|+|v_\mu(t)|\to+\infty$ as $|t|\to T_\mu$.

Let us denote 
\[
H_\mu(t)=H(t,u_\mu(t),v_\mu(t)), \quad t\in(-T_\mu,T_\mu).
\]
A simple computation implies
\[
\frac{d}{dt}H_\mu(t)=\frac{d}{dt}\Big[\cosh(t)^{-\frac1{m-1}}\Big]\frac{m-1}{2m}(u_\mu^2+v_\mu^2)^{\frac{m}{m-1}}\leq0, \quad \forall t\geq0
\]
so that the energy $H_\mu$ is non-increasing along the solution $(u_\mu,v_\mu)$, on $[0,T_\mu)$. However, since we have $|u_\mu(t)|+|v_\mu(t)|\to+\infty$ as $t\to T_\mu$, we find
\[
H_\mu(t)\geq-\frac{m-2}2u_\mu(t)v_\mu(t)+\frac{m-1}{2m}\cosh(T_\mu)^{-\frac1{m-1}}(u_\mu(t)^2+v_\mu(t)^2)^{\frac m{m-1}}\to+\infty
\]
as $t\to T_\mu$, which is absurd. Hence we have $u_\mu$ and $v_\mu$ are globally defined on $\R$.
\end{proof}

In what follows, we state some basic properties for solutions of \eqref{reduced-SY4}.

\begin{Lem}\label{ODE2-lemma2}
Given $\mu>0$, then the following holds:
\begin{itemize}
\item If, for some $t_0\neq0$, we have $u_\mu(t_0)=0$, then $v_\mu(t_0)\neq0$ and $u_\mu'(t_0)\neq0$.
\item If, for some $t_0>0$, we have $v_\mu(t_0)=0$, then $u_\mu(t_0)\neq0$ and $v_\mu'(t_0)\neq0$.
\end{itemize}
Moreover, both $u_\mu$ and $v_\mu$ can not change sign infinitely many times in a bounded interval $[-T,T]$.
\end{Lem}
\begin{proof}
Observe that the only rest point of system \eqref{reduced-SY4} is $(0,0)$. Furthermore, for $t_0\neq0$, the Cauchy problem for \eqref{reduced-SY4} is locally well-posed for any initial datum $(u(t_0),v(t_0))\in\R^2$, for both $t>t_0$ and $t<t_0$. Thus, a rest point cannot be reached in a finite time. 

In order to see that both $u_\mu$ and $v_\mu$ can only change sign a finite number of times in a bounded interval $[-T,T]$, we assume by contradiction that there exists $\{t_j^u\}$ and $\{t_j^v\}$ in $[-T,T]$ such that $t_j^u\to T_u$ and $t_j^v\to T_v$ as $j\to\infty$,  $u_\mu(t_j^u)=v_\mu(t_j^v)=0$ for all $j$, and $u_\mu$ (resp. $v_\mu$) changes sign a finite number of times on $[-|T_u|+\de,|T_u|-\de]$ (resp. $[-|T_v|+\de,|T_v|-\de]$) for any $\de>0$. 

If $|T_u|<|T_v|$, then $v_\mu$ will not change sign in a left neighborhood of $|T_u|$ and in a right neighborhood of $-|T_u|$. Then the first equation in \eqref{reduced-SY4} implies that $u_\mu'(t_j^u)$ has the same sign as $v_\mu$, which is impossible. Hence $|T_u|\geq |T_v|$. Similarly, one obtains $|T_v|\geq |T_u|$. Therefore $|T_u|=|T_v|$. Moreover, it can not happen that $T_u=-T_v$ while $u_\mu$ (resp. $v_\mu$) keeps a definite sign around $T_v$ (resp. $T_u$). Therefore, we must have $T_u=T_v=T_0$. In particular, we have $u_\mu(T_0)=v_\mu(T_0)=0$, which is also impossible.
\end{proof}

\begin{Lem}\label{ODE2-lemma3}
Given $\mu>0$. If $(u_\mu,v_\mu)$ is a bounded solution, i.e., $|u_\mu(t)|+|v_\mu(t)|\leq M$ for all $t\in\R$ and some $M>0$, then $(u_\mu,v_\mu)\to(0,0)$ as $|t|\to+\infty$. 
\end{Lem}
\begin{proof}
By symmetry, we only need to prove the result for $t\to+\infty$. Multiplying by $u_\mu$ (resp. $v_\mu$) the equations in \eqref{reduced-SY4}, we have
\[
\left\{
\aligned
uu'&=\cosh(t)^{-\frac1{m-1}}(u_\mu^2+v_\mu^2)^{\frac1{m-1}}u_\mu v_\mu-\frac{m-2}2u_\mu^2, \\
-vv'&=\cosh(t)^{-\frac1{m-1}}(u_\mu^2+v_\mu^2)^{\frac1{m-1}}u_\mu v_\mu-\frac{m-2}2v_\mu^2.
\endaligned
\right.
\]
Thus we need to show that $u_\mu(t)^2+v_\mu(t)^2\to0$ as $t\to+\infty$.

Suppose by contradiction that, for arbitrary small $\vr>0$, there exists $t_0>0$ large such that 
\[
\cosh(t_0)^{-\frac1{m-1}}M^{\frac{m}{m-1}}\leq 2\vr \quad \text{and} \quad
u_\mu(t_0)^2+v_\mu(t_0)^2\geq 2\de_0,
\]
for some $\de_0>0$. Since
\[
\frac12 (u_\mu^2)'\leq \vr -\frac{m-2}2 u_\mu^2,
\]
we find
\[
u_\mu(t)^2\leq \frac{2\vr}{m-2}-\frac{2\vr}{m-2}e^{(m-2)(t_0-t)}+u_\mu(t_0)^2e^{(m-2)(t_0-t)}.
\]
Therefore, by enlarging $t_0$, we can assume without loss of generality that $v_\mu(t_0)^2>\de_0$. And hence, we obtain
\[
-\frac12(v_\mu^2)'\leq \vr -\frac{m-2}2v_\mu^2,
\]
which implies
\[
v_\mu(t)^2\geq\frac{2\vr}{m-2}-\frac{2\vr}{m-2}e^{(m-2)(t-t_0)}+v_\mu(t_0)^2e^{(m-2)(t-t_0)}.
\]
By taking $\vr<\frac{m-2}2\de_0$, we have $v_\mu(t)^2\to+\infty$ as $t\to+\infty$. This contradicts the boundedness of $v_\mu$.
\end{proof}

\begin{Rem}
From the above result, we can conclude that, if there exists $t_0>0$ such that $H_\mu(t_0)\leq0$, the corresponding solution $(u_\mu,v_\mu)$ must be unbounded as $t\to\pm\infty$ (since the energy $H_\mu(t)=H(t,u_\mu(t),v_\mu(t))$ is decreasing).
\end{Rem}

\begin{Lem}\label{ODE2-lemma-unbddness}
Let $\mu>0$. If $(u_\mu,v_\mu)$ is a solution such that $\lim_{|t|\to+\infty}H_\mu(t)\in[-\infty,0)$. Then $u_\mu(t)^2+v_\mu(t)^2=O(\cosh(t)) $ as $|t|\to+\infty$.
\end{Lem}
\begin{proof}
Since $H_\mu(t)$ is decreasing, we can take $t_0>0$ such that $H_\mu(t_0)\leq0$ and
\[
\aligned
0\geq H_\mu(t) \geq \frac{m-1}{2m}\cosh(t)^{-\frac1{m-1}}(u_\mu(t)^2+v_\mu(t)^2)^{\frac m{m-1}}-\frac{m-2}4 (u_\mu(t)^2+v_\mu(t)^2)
\endaligned
\]
for all $t\geq t_0$
Notice that $u_\mu(t)^2+v_\mu(t)^2$ can not reach $0$ in a finite time, we soon have
\[
u_\mu(t)^2+v_\mu(t)^2\leq c_m\cosh(t)
\]
for all $t\geq t_0$ and $c_m>0$ depends only on $m$. 
\end{proof}

\begin{Lem}\label{ODE2-lemma4}
Let $(u_\mu,v_\mu)$ be a solution of \eqref{reduced-SY4} such that $v_\mu$ changes sign a finite number of times on $\R$, then there exists $T>0$ such that $u_\mu(t)v_\mu(t)>0$ for all $|t|\geq T$.
\end{Lem}
\begin{proof}
Since $v_\mu$ changes sign a finite number of times on $\R$, we suppose without loss of generality that $v_\mu(t)>0$ for all $t\geq T_1$, some $T_1>0$.

Assume, by contradiction, that $u_\mu(t)<0$ for all $t>T_1$. Then the second equation of \eqref{reduced-SY4} implies that $v_\mu'(t)>0$ for $t>T_1$, that is, $v_\mu(t)$ is increasing for $t>T_1$. Hence we have
\[
\lim_{t\to+\infty}v_\mu(t)=v_\infty\in(0,+\infty].
\]
Notice that, by the second equation again, we have
\[
v'\geq \frac{m-2}2 v \quad \text{for } t\geq T_1.
\]
We deduce that 
\[
v_\mu(t)\geq v_\mu(T_1)e^{\frac{m-2}2(t-T_1)} \quad \text{for } t\geq T_1.
\]
Hence $v_\infty=+\infty$. However, since $u_\mu$ and $v_\mu$ have opposite sign, we find
\[
H_\mu(t)\geq\frac{m-1}{2m}\cosh(t)^{-\frac1{m-1}}(u_\mu(t)^2+v_\mu(t)^2)^{\frac m{m-1}}>\frac{m-1}{2m}\cosh(t)^{-\frac1{m-1}}v_\mu(t)^{\frac{2m}{m-1}}\to+\infty
\]
as $t\to+\infty$, which is impossible.

Let $t_0\geq T_1$ be such that $u_\mu(t_0)=0$. Then, it follow from the first equation of \eqref{reduced-SY4} that $u_\mu'(t_0)>0$. If there exists $\hat t_0>t_0$ such that $u_\mu(\hat t_0)=0$ and $u_\mu(t)>0$ on $(t_0,\hat t_0)$, we soon derive that $u_\mu'(t)<0$ in a left neighborhood of $\hat t_0$. Thus, by the the first equation of \eqref{reduced-SY4} again, we get $v_\mu(\hat t_0)\leq0$. This is impossible since we have assumed $v_\mu(t)>0$ for all $t>T_1$. Therefore, by taking $T> t_0$, we conclude $u_\mu(t)>0$ for all $t\geq T$.
\end{proof}

\begin{Cor}\label{ODE2-corollary1}
Let $(u_\mu,v_\mu)$ be a solution of \eqref{reduced-SY4} such that $u_\mu$ changes sign a finite number of times on $\R$, then there exists $T>0$ such that $u_\mu(t)v_\mu(t)>0$ for all $|t|\geq T$.
\end{Cor}
\begin{proof}
Suppose that we have $u_\mu(t)>0$ for all $t\geq T$, some $T>0$. By Lemma \ref{ODE2-lemma2} and \ref{ODE2-lemma4}, we can not have $v_\mu(t)<0$ for all $t>T$. 

Suppose that there exists $t_0>T_1$ such that $v_\mu(t_0)=0$. Then $v_\mu'(t_0)<0$ and $v_\mu$ enters to negative values, and can not have further zeros. In fact, if there is $\hat t_0>t_0$ such that $v_\mu(\hat t_0)=0$ and $v_\mu(t)<0$ on $(t_0,\hat t_0)$. We will have $v_\mu'(\hat t_0)\geq0$, which is impossible. Then we obtain a contradiction with Lemma \ref{ODE2-lemma4}.

\end{proof}

\begin{Cor}\label{ODE2-lemma-exponential-decay}
Let $(u_\mu,v_\mu)$ be a bounded solution of \eqref{reduced-SY4} such that $v_\mu$ (or $u_\mu$) changes sign a finite number of times on $\R$, then
\[
 u_\mu(t)^2+v_\mu(t)^2=O(e^{-(m-2)t})
\]
as $|t|\to+\infty$.
\end{Cor}
\begin{proof}
By virtue of Lemma \ref{ODE2-lemma4} and Corollary \ref{ODE2-corollary1}, we can take $T>1$ large enough such that $u_\mu(t)v_\mu(t)>0$ for all $t\geq T$. Then, it can be derived from \eqref{reduced-SY4} that
\[
-(u_\mu^2+v_\mu^2)''+(m-2)^2(u_\mu^2+v_\mu^2)
=4(m-2)\cosh(t)^{-\frac1{m-1}}(u_\mu^2+v_\mu^2)^{\frac1{m-1}}u_\mu v_\mu.
\]
Hence, from the boundedness of $u_\mu$ and $v_\mu$, we have
\begin{\equ}\label{XX1}
\left\{
\aligned
&-(u_\mu^2+v_\mu^2)''+(m-2)^2(u_\mu^2+v_\mu^2)>0 \\
&-(u_\mu^2+v_\mu^2)''+(m-2)^2(u_\mu^2+v_\mu^2)\leq \de e^{-\frac1{m-1}t}(u_\mu^2+v_\mu^2)
\endaligned
\right.
\end{\equ}
for $t$ sufficiently large, where $\de>0$ is a constant.

Let $\Ga_1(t)=e^{-(m-2)t}$ and $\Ga_2(t)=\arctan(t)e^{-(m-2)t}$, for $t>0$. One checks easily that
\[
-\Ga_1''+(m-2)^2\Ga_1=0 \quad \text{and} \quad 
-\Ga_2''+(m-2)^2\Ga_2 \geq\frac{2(m-2)}{1+t^2}e^{-(m-2)t}.
\]
By taking $C_1,C_2>0$ such that 
\[
C_1\Ga_1(T_0)\leq u_\mu(T_0)^2+v_\mu(T_0)^2 \leq C_2 \Ga_2(T_0),
\]
for some $T_0>T$, we find 
\[
\left\{
\aligned
&-(u_\mu^2+v_\mu^2-C_1\Ga_1)''+(m-2)^2(u_\mu^2+v_\mu^2-C_1\Ga_1)>0 ,  \\
&-(u_\mu^2+v_\mu^2-C_2\Ga_2)''+\Big[ (m-2)^2 - \frac{2(m-2)}{(1+t^2)\arctan(t)}\Big](u_\mu^2+v_\mu^2-C_2\Ga_2)<0,
\endaligned
\right.
\]
for all $t>T_0$.
Then, by the comparison principle, we have
\[
C_1\Ga_1(t) \leq u_\mu(t)^2+v_\mu(t)^2 \leq C_2\Ga_2(t),
\]
for all $t>T_0$, which completes the proof. 
\end{proof}

\begin{Lem}\label{ODE2-lemma5}
Let $(u_\mu,v_\mu)$ be a solution of \eqref{reduced-SY4} such that $v_\mu$ changes sign a finite number of times on $\R$. If $H_\mu(t)=H(t,u_\mu(t),v_\mu(t))>0$ for all $t>0$, then $H_\mu(t)\leq C e^{-c|t|}$ as $t\to\pm\infty$, for some constants $C,c>0$ possibly depending on $\mu$.
\end{Lem} 
\begin{proof}
We only prove the result for $t\to+\infty$. Note that
\[
\aligned
\frac{d}{dt}H_\mu(t)&=\frac{d}{dt}\Big[ \cosh(t)^{-\frac1{m-1}} \Big]\frac{m-1}{2m}(u_\mu^2+v_\mu^2)^{\frac m{m-1}} \\[0.3em]
&=-\frac1{2m}\cosh(t)^{-\frac1{m-1}}\frac{e^t-e^{-t}}{e^t+e^{-t}}(u_\mu^2+v_\mu^2)^{\frac m{m-1}} \\[0.3em] 
&\leq-\frac{1-\de}{2m}\cosh(t)^{-\frac1{m-1}}(u_\mu^2+v_\mu^2)^{\frac m{m-1}} \\[0.3em] 
&\leq -\frac{1-\de}{m-1} H_\mu(t), \quad \text{for } t \geq T_\de ,
\endaligned
\]
where $\de>0$ can be fixed arbitrarily small and the last inequality comes from Lemma \ref{ODE2-lemma4}. Therefore, we have
\[
H_\mu(t)\leq H_\mu(T_\de)e^{-\frac{1-\de}{m-1}t}
\]
for all $t\geq T_\de$, which completes the proof.
\end{proof}

Now, for $\mu>0$ and $(u_\mu,v_\mu)$ the corresponding solution of \eqref{reduced-SY4}, we introduce the sets $A_k$, $B_k$ and $I_k$ defined for $k\in\N\cup\{0\}$ by
\[
A_k=\Big\{ \mu>0:\, v_\mu \text{ changes sign } k \text{ times on } (0,+\infty) \text{ and } \lim_{|t|\to+\infty}H_\mu(t)<0 \Big\},
\]
\[
B_k=\Big\{ \mu>0:\,  v_\mu \text{ changes sign } k \text{ times on } (0,+\infty),\ H_\mu(t)>0 \text{ and } (u_\mu,v_\mu) \text{ is unbounded} \Big\},
\]
\[
I_k=\Big\{ \mu>0:\, v_\mu \text{ changes sign } k \text{ times on } (0,+\infty),\ H_\mu(t)>0 \text{ and } (u_\mu,v_\mu) \text{ is bounded}  \Big\}.
\]
Notice that $(0,0)$ is a hyperbolic equilibrium point of the Hamiltonian energy $H(t,\cdot,\cdot)$ for any $t\in\R$. It is, then, immediate to see that $A_0\neq\emptyset$ as it includes the interval $(0,\frac{\sqrt2}2]$, since
\[
H(0,\mu,\mu)<0 \quad \text{for all } \mu\in\Big(0,\frac{\sqrt2}2  \,\Big].
\]
As we will see later, tracking the sign changes of the solutions is crucial for the proof of Theorem \ref{main thm ODE2}. The main idea is to study the stratified structure of the solutions. This will be done by checking their topology and boundedness. The boundedness, allows us to track the $\sup$ of $A_{k}$ and $I_{k}$ allowing us to prove that all the sets $A_{k}$ are not empty. As we will see below, the idea of tracking the signs coming from a limiting problem with explicit solutions and infinitely many sign changes. This property will allow us to prove boundedness of the desired sets.

Let us start first by discarding the sets $B_{k}$:

\begin{Lem}\label{ODE2-lemma6}
	$B_k=\emptyset$ for all $k\in\N\cup\{0\}$.
\end{Lem}
\begin{proof}
Suppose to the contrary that $B_k\neq\emptyset$ for some $k$. Let $\mu\in B_k$ and $(u_\mu,v_\mu)$ be the corresponding solution. Then, by substituting $(u_\mu,v_\mu)$ into Eq. \eqref{reduced-SY4}, we obtain
\begin{\equ}\label{ODE2-2}
	\left\{
	\aligned
	u_\mu'v_\mu&=\cosh(t)^{-\frac1{m-1}}(u_\mu^2+v_\mu^2)^{\frac1{m-1}}v_\mu^2-\frac{m-2}2 u_\mu v_\mu , \\
	-u_\mu v_\mu'&=\cosh(t)^{-\frac1{m-1}}(u_\mu^2+v_\mu^2)^{\frac1{m-1}}u_\mu^2-\frac{m-2}2 u_\mu v_\mu .
	\endaligned
	\right.
\end{\equ}
This gives 
\[
\aligned
u_\mu' v_\mu -u_\mu v_\mu'&=\cosh(t)^{-\frac1{m-1}}(u_\mu^2+v_\mu^2)^{\frac m{m-1}}-(m-2)u_\mu v_\mu \\ 
&= \frac{2m}{m-1} H_\mu(t)+\frac{m-2}{m-1}u_\mu v_\mu > \frac{m-2}{m-1}u_\mu v_\mu,
\endaligned
\]
for all $t$. By Lemma \ref{ODE2-lemma4}, for $t$ large enough, we can divide the above inequality by $u_\mu v_\mu$ to get
\[
(\ln u_\mu -\ln v_\mu)'> \frac{m-2}{m-1},
\]
where we have assumed without loss of generality that $u_\mu(t)>0$ and $v_\mu(t)>0$ for $t$ large. Hence we have 
\begin{\equ}\label{X1}
\frac{u_\mu(t)}{v_\mu(t)}\geq C e^{\frac{m-2}{m-1}t}
\end{\equ}
for some constant $C>0$. And therefore, there exists $T>0$ such that $u_\mu(t)> v_\mu(t)$ for all $t> T$. Now, by \eqref{ODE2-2}, we have
\[
u_\mu'v_\mu+u_\mu v_\mu'=\cosh(t)^{-\frac1{m-1}}(u_\mu^2+v_\mu^2)^{\frac1{m-1}}(v_\mu^2-u_\mu^2)<0
\]
for $t>T$, that is, $u_\mu v_\mu$ is decreasing for all large $t$. 

Assume that $u_\mu(t)v_\mu(t)\to a_\infty\in[0,+\infty)$ as $t\to\infty$. By Lemma \ref{ODE2-lemma4} and \ref{ODE2-lemma5}, we have
\[
\frac{m-1}{2m}\cosh(t)^{-\frac1{m-1}}(u_\mu(t)^2+v_\mu(t)^2)^{\frac{m}{m-1}}\to \frac{m-2}2 a_\infty
\]
as $t\to\infty$. Therefore, for arbitrary small $\vr>0$, there exists $T_\vr>0$ such that
\[
\left\{
\aligned
u_\mu' &\leq \vr -\frac{m-2}2u_\mu \\
-v_\mu' &\leq \vr-\frac{m-2}2v_\mu
\endaligned
\right.
\]
for all $t\geq T_\vr$. This implies
\[
u_\mu(t)\leq \frac{2\vr}{m-2}-\frac{2\vr}{m-2}e^{\frac{m-2}2(T_\vr-t)}+u_\mu(T_\vr)e^{\frac{m-2}2(T_\vr-t)}
\]
and
\[
v_\mu(t)\geq \frac{2\vr}{m-2}-\frac{2\vr}{m-2}e^{\frac{m-2}2(t-T_\vr)}+v_\mu(T_\vr)e^{\frac{m-2}2(t-T_\vr)}
\]
for all $t\geq T_\vr$. Since $\mu\in B_k$, we have $|u_\mu(t)|+|v_\mu(t)|$ is unbounded as $|t|\to+\infty$. Hence, by fixing $\vr>0$ suitably small, we find 
\[
v_\mu(t)\sim e^{\frac{m-2}2t} \quad \text{and} \quad u_\mu(t)\to0
\]
as $t\to+\infty$, this contradicts \eqref{X1}.
\end{proof}

\begin{Lem}\label{ODE2-lemma7}
There exists constants $C_0>0$  such that, if for some $T>1$,
\begin{itemize}
	\item[$(1)$] $H_\mu(T)\leq C_0$;
	
	\item[$(2)$] $u_\mu(T)v_\mu(T)>0$;
	
	\item[$(3)$] $v_\mu$ changes sign $k$ times on $[0,T]$;
\end{itemize}
then $\mu\in A_k\cup I_k\cup A_{k+1}$.
\end{Lem}
\begin{proof}
Suppose that $\mu\not\in A_k\cup I_k$, it remains to show that $\mu\in A_{k+1}$. Without loss of generality, let us assume that $u_\mu(T)>0$ and $v_\mu(T)>0$. Set
\[
\widetilde T=\inf \big\{ t>T: u_\mu(t)\leq 0 \big\}\in(T,+\infty].
\]

If $\widetilde T=+\infty$, we have $v_\mu$ changes sign at most once in $(T,+\infty)$. Indeed, as long as $u_\mu>0$, the second equation of \eqref{reduced-SY4} implies that $v_\mu'<0$ whenever $v_\mu$ vanishes. Therefore, $v_\mu$ can not change sign more than once. If $v_\mu$ does not change sign on $(T,+\infty)$, we have $\mu\in A_k\cup I_k$, which is absurd. However, if $v_\mu$ does change sign once in $(T,+\infty)$, we have $u_\mu(t)v_\mu(t)<0$ for all large $t$. This contradicts Lemma \ref{ODE2-lemma4}.  Therefore, we have $\widetilde T<+\infty$ and $u_\mu(\widetilde T)=0$.

\begin{claim}\label{claim1}
$v_\mu$ changes sign exactly once in $(T,\widetilde T)$.
\end{claim}
In fact, by rewriting the second equation of \eqref{reduced-SY4}, we have
\[
\Big(v_\mu(t)e^{-\frac{m-2}2t}\Big)'=-\cosh(t)^{-\frac1{m-1}}(u_\mu(t)^2+v_\mu(t)^2)^{\frac1{m-1}}u_\mu(t)e^{-\frac{m-2}2t}<0
\]
for $t\in(T,\widetilde T)$. If $v_\mu$ stays positive on $(T,\widetilde T)$, by Lemma \ref{ODE2-lemma2}, we have $u_\mu'\geq0$ on a left neighborhood of $\widetilde T$, which is impossible.

\medskip

To proceed, let us set $f_\mu=(u_\mu-v_\mu)/\sqrt2$ and $g_\mu=(u_\mu+v_\mu)/\sqrt2$. Then $(f_\mu,g_\mu)$ satisfies the following system
\begin{\equ}\label{ODE3}
\left\{
\aligned
f'&=\cosh(t)^{-\frac1{m-1}}(f^2+g^2)^{\frac1{m-1}}g-\frac{m-2}2g, \\
-g'&=\cosh(t)^{-\frac1{m-1}}(f^2+g^2)^{\frac1{m-1}}f+\frac{m-2}2f,
\endaligned
\right.
\end{\equ}
with Hamiltonian energy
\[
\widehat H(t,f,g)=\frac{m-2}4f^2-\frac{m-2}4g^2+\frac{m-1}{2m}\cosh(t)^{-\frac1{m-1}}(f^2+g^2)^{\frac m{m-1}}.
\]
Clearly, we have $H_\mu(t)=\widehat H(t,f_\mu,g_\mu)$ for $t\in\R$. And, by Claim \ref{claim1},  we can make $T$ slightly larger so that $u_\mu>v_\mu$ on $[T,\widetilde T]$. That is, we have $f_\mu>0$ on $[T,\widetilde T]$, $g_\mu(T)>0$,  $g_\mu(\widetilde T)<0$ and $g_\mu$ changes sign once in $(T,\widetilde T)$.

In what follows, we are going to prove that $f_\mu$ stays positive on $[T,+\infty)$. Then the second equation in \eqref{ODE3} shows that $g_\mu'<0$ for all $t\geq T$. And hence $\mu\not\in I_j$ for any $j\in\N\cup\{0\}$. 
In this case, we have $f_\mu(t)>0$ and $g_\mu(t)<0$ for all $t\geq \widetilde T$, which implies $v_\mu(t)<0$ for $t\in[\widetilde T,+\infty)$. That is, $v_\mu$ changes sign exactly once on $(T,+\infty)$. Therefore $\mu\in A_{k+1}$.

Suppose, by contradiction, that there exists $\widehat T>\widetilde T$ such that $f_\mu(\widehat T)=0$ and $f_\mu>0$ on $[T,\widehat T)$. Then, the second equation in \eqref{ODE3} implies that $g_\mu$ is decreasing on $[T,\widehat T]$. And hence, $g_\mu(\widehat T)<g_\mu(\widetilde T)<0$. Then, we only need to consider the situation $H_\mu(\widehat T)>0$, since the condition $H_\mu(\widehat T)\leq0$ will immediately trap the solution $(u_\mu,v_\mu)$ in the third quadrant of $(u,v)$-plane for $t> \widehat T$, and leads us to have $\mu\in A_{k+1}$.

In the case $H_\mu(\widehat T)>0$, by $f_\mu(\widehat T)=0$ and $g_\mu(\widehat T)<0$, we have
\[
g_\mu(\widehat T)< -\Big(\frac{m(m-2)}{2(m-1)}\Big)^{\frac{m-1}2}\cosh(\widehat T)^{\frac12}.
\]
Let $T<T_1<T_2<\widehat T$ be such that 
\[
\frac{m-1}{2m}\cosh(\widehat T)^{-\frac1{m-1}}g_\mu(T_1)^{\frac{2m}{m-1}}-\frac{m-2}4g_\mu(T_1)^2=-C_0
\]
and
\[
\frac{m-1}{2m}\cosh(\widehat T)^{-\frac1{m-1}}g_\mu(T_2)^{\frac{2m}{m-1}}-\frac{m-2}4g_\mu(T_2)^2=0.
\]
By assuming $C_0$ suitably small, such $T_1$ and $T_2$ always exist, and we can have that $g_\mu(\widehat T)< g_\mu(T_2)<g_\mu(T_1)<g_\mu(T_2)/2<0$. In fact, by setting
\[
F(s)=\frac{m-1}{2m}\cosh(\widehat T)^{-\frac1{m-1}}|s|^{\frac{2m}{m-1}}-\frac{m-2}4|s|^2, \quad s\in\R
\]
we have $g_\mu(T_2)$ is nothing but the vanishing point of $F$ in the negative line, i.e.,
\begin{\equ}\label{gmuT2}
	g_\mu(T_2)=-\Big(\frac{m(m-2)}{2(m-1)}\Big)^{\frac{m-1}2}\cosh(\widehat T)^{\frac12},
\end{\equ}
and $g_\mu(T_1)$ is the smallest point such that $F=-C_0$. Then, use the fact $H_\mu(t)\leq C_0$ for all $t>T$, we have
\[
\frac{m-2}{4}f_\mu(t)^2\leq C_0 - F(g_\mu(t))\leq 2C_0
\]
for $t\in[T_1,T_2]$. Hence, we deduce
\begin{\equ}\label{fmu-bdd}
	0<f_\mu(t)\leq \de_0:=\sqrt{\frac{8C_0}{m-2}}
\end{\equ}
for $t\in[T_1,T_2]$. Notice that
\[
F'(g_\mu(T_2))=-\frac1{m-1}\Big(\frac{m}{m-1}\Big)^{\frac{m-1}2}\Big(\frac{m-2}2\Big)^{\frac{m+1}2}\cosh(\widehat T)^{\frac12}<0
\]
and
\[
F''(g_\mu(T_2))=\frac{m-2}2\Big( \frac{m(m+1)}{(m-1)^2}-1\Big)>0.
\]
By using the second equation in \eqref{ODE3} and \eqref{fmu-bdd}, we find
\begin{\equ}\label{T2-T1}
\aligned
\frac{C_0}{F'(g_\mu(T_2))}&>g_\mu(T_2)-g_\mu(T_1)=\int_{T_1}^{T_2}g_\mu'(t)dt \\
&\geq-\int_{T_1}^{T_2}\Big[ \big( \de_0^2 + g_\mu(T_2)^2 \big)^{\frac1{m-1}}\de_0 + \frac{m-2}2\de_0 \Big]dt \\
&\geq - C_m g_\mu(T_2)^{\frac{2}{m-1}}\de_0(T_2-T_1)
\endaligned
\end{\equ}
where $C_m>0$ depends only on $m$ (since we have assumed $C_0$ is small). On the other hand, we have 
\[
\aligned
\frac{d}{dt}H_\mu(t)&=-\frac1{2m}\cosh(t)^{-\frac1{m-1}}\frac{e^t-e^{-t}}{e^t+e^{-t}}(f_\mu(t)^2+g_\mu(t)^2)^{\frac m{m-1}} \\
&\leq -\frac1{2m}\frac{e-e^{-1}}{e+e^{-1}}\cosh(\widehat T)^{-\frac1{m-1}}g_\mu(T_1)^{\frac{2m}{m-1}} \\
&\leq -c_m\cosh(\widehat T)^{-\frac1{m-1}}g_\mu(T_2)^{\frac{2m}{m-1}} 
\endaligned
\]
for $t\in[T_1,T_2]$, where in the last inequality we used $|g_\mu(T_1)|>\frac12|g_\mu(T_2)|$ and
\[
c_m=\frac1{2m}\Big(\frac12\Big)^{\frac{2m}{m-1}}\frac{e-e^{-1}}{e+e^{-1}}.
\] 
Hence, by \eqref{T2-T1}, we obtain
\[
\aligned
H_\mu(T_2)-H_\mu(T_1)&=\int_{T_1}^{T_2}\frac{d}{dt}H_\mu(t) dt
\leq -c_m\cosh(\widehat T)^{-\frac1{m-1}}g_\mu(T_2)^{\frac{2m}{m-1}}(T_2-T_1) \\
&\leq \frac{c_m\cosh(\widehat T)^{-\frac1{m-1}}g_\mu(T_2)^{\frac{2m}{m-1}} C_0}{C_m F'(g_\mu(T_2))g_\mu(T_2)^{\frac2{m-1}}\de_0}\\
&= -\widetilde C_m\cosh(\widehat T)^{\frac12-\frac1{m-1}} \sqrt{C_0}<-C_0
\endaligned
\]
provided that $m\geq3$ and $C_0$ is small enough. This implies $H_\mu(T_2)\leq 0$ reaching a contradiction, and the proof is hereby completed.
\end{proof}

The next lemma provides the main properties of the sets $A_k$ and $I_k$.

\begin{Lem}\label{ODE2-lemma8}
For all $k\in\N\cup\{0\}$, we have
\begin{itemize}
	\item[$(1)$] $A_k$ is an open set;
	
	\item[$(2)$] if $\mu\in I_k$, then there exists $\vr>0$ such that $(\mu-\vr,\mu+\vr)\subset A_k\cup I_k\cup A_{k+1}$;
	
	\item[$(3)$] if $A_k\neq\emptyset$ and is bounded, then $\sup A_k\in I_k$;
	
	\item[$(4)$] if both $A_k$ and $I_k$ are bounded, set $\mu=\sup I_k$, then there exists $\vr>0$ such that $(\mu,\mu+\vr)\subset A_{k+1}$.
\end{itemize}
\end{Lem}
\begin{proof}
$(1)$ is quite obvious, since it comes from the continuity of the solutions $(u_\mu,v_\mu)$ with respect to the initial datum. 

To see $(2)$, we fix $\mu\in I_k$. Then we have $H_\mu(t)\to0$ as $|t|\to+\infty$. Given $C_0$ as in Lemma \ref{ODE2-lemma7}, there exists $T>1$ such that $H_\mu(T)<C_0$, $u_\mu(T)v_\mu(T)>0$ and $v_\mu$ changes sign $k$ times on $[0,T]$.  The continuity of the solution $(u_\mu,v_\mu)$ with respect to $\mu$ implies that the same holds for an initial datum $\tilde\mu\in(\mu-\vr,\mu+\vr)$ for $\vr>0$ small. Then the conclusion follows by Lemma \ref{ODE2-lemma7}.

To check $(3)$, let us set $\mu=\sup A_k$ and take a sequence $\{\mu_j\}\subset A_k$ such that $\mu_j\nearrow\mu$ as $j\to+\infty$. If we suppose that $\mu\in A_l$ for some $l$, then $(1)$ suggests that $\mu_j\in A_l$ for $j$ large. Hence we have $l=k$. This implies $\mu\in A_k$ which is absurd since $A_k$ is an open set. Notice that, by the continuity property of the solutions, the corresponding $v_\mu$ can change sign only a finite number of times on $(0,+\infty)$. Therefore we must have that $\mu\in I_s$ for some $s$. By $(2)$, we have $(\mu-\vr,\mu+\vr)\subset A_s\cup I_s\cup A_{s+1}$. This implies $s=k$.

Finally, to see $(4)$, we first observe that $\mu=\sup I_k\in I_k$. Indeed, let $\{\mu_j\}\in I_k$ be such that $\mu_j\nearrow\mu$ as $j\to+\infty$, we have $\mu\not\in A_l$ for any $l\in\N\cup\{0\}$. This is because $A_l$ is an open set. Then, arguing similarly as in $(3)$, we get that $\mu\in I_k$ as claimed. Now, by $(2)$, we have 
$(\mu,\mu+\vr)\subset A_k\cup A_{k+1}$ for some $\vr>0$. Since we have assumed the boundedness of $A_k$, we find $\sup A_k\leq \mu$. Thus $(\mu,\mu+\vr)\subset A_{k+1}$.
\end{proof}

Our next result is the boundedness property of the sets $A_k$ and $I_k$.

\begin{Prop}\label{ODE2-proposition1}
$A_k\cup I_k$ is bounded for each $k\in\N\cup\{0\}$.
\end{Prop}

Before prove Proposition \ref{ODE2-proposition1}, let us do some preparations. Denoted by $\vr=\mu^{-1}>0$, we consider the following rescaling
\[
\left\{
\aligned
U_\vr(t)=\vr u_\mu\big(\vr^{\frac2{m-1}}t\big), \\
V_\vr(t)=\vr v_\mu\big(\vr^{\frac2{m-1}}t\big).
\endaligned
\right.
\]
We find the system for $(U_\vr,V_\vr)$ is
\begin{\equ}\label{ODE2-rescaled}
\left\{
\aligned
U_\vr'&=\cosh\big(\vr^{\frac2{m-1}}t\big)^{-\frac1{m-1}}(U_\vr^2+V_\vr^2)^{\frac1{m-1}}V_\vr -\vr^{\frac2{m-1}}\frac{m-2}2U_\vr \\
-V_\vr'&=\cosh\big(\vr^{\frac2{m-1}}t\big)^{-\frac1{m-1}}(U_\vr^2+V_\vr^2)^{\frac1{m-1}}U_\vr -\vr^{\frac2{m-1}}\frac{m-2}2V_\vr
\endaligned
\right.
\end{\equ}
together with the initial datum $U_\vr(0)=V_\vr(0)=1$. The limiting problem associated to Eq.~\eqref{ODE2-rescaled} is
\begin{\equ}\label{ODE2-rescaled-limit}
	\left\{
	\aligned
	U_0'&=(U_0^2+V_0^2)^{\frac1{m-1}}V_0  \\
	-V_0'&=(U_0^2+V_0^2)^{\frac1{m-1}}U_0
	\endaligned
	\right.
\end{\equ}
with $U_0(0)=V_0(0)=1$.

\begin{Lem}\label{ODE2-lemma9}
There holds
\[
(U_\vr,V_\vr)\to (U_0,V_0) \quad \text{as } \vr\to0
\]
uniformly on $[0,T]$, for all $T>0$, where $(U_0,V_0)$ is the solution to Eq. \eqref{ODE2-rescaled-limit}.
\end{Lem}
\begin{proof}
First of all, we have \eqref{ODE2-rescaled} is equivalent to
\begin{\equ}\label{X2}
	\left\{
	\aligned
	U_\vr(t)&=1+\int_0^t \Big[ \cosh\big(\vr^{\frac2{m-1}}s\big)^{-\frac1{m-1}}(U_\vr^2+V_\vr^2)^{\frac1{m-1}}V_\vr -\vr^{\frac2{m-1}}\frac{m-2}2U_\vr \Big] ds \\
	V_\vr(t)&=1-\int_0^t \Big[ \cosh\big(\vr^{\frac2{m-1}}s\big)^{-\frac1{m-1}}(U_\vr^2+V_\vr^2)^{\frac1{m-1}}U_\vr -\vr^{\frac2{m-1}}\frac{m-2}2V_\vr \Big] ds
	\endaligned
	\right.
\end{\equ}
and, similarly, \eqref{ODE2-rescaled-limit} is equivalent to
\begin{\equ}\label{X3}
	\left\{
	\aligned
	U_0(t)&=1+\int_0^t(U_0^2+V_0^2)^{\frac1{m-1}}V_0 \, ds, \\
	V_0(t)&=1-\int_0^t(U_0^2+V_0^2)^{\frac1{m-1}}U_0 \, ds .
	\endaligned
	\right.
\end{\equ}

The Hamiltonian energy associated to  \eqref{ODE2-rescaled} is  given by 
\[
H_\vr(t,U,V)=-\vr^{\frac2{m-1}}\frac{m-2}{2}UV+\frac{m-1}{2m}\cosh\big(\vr^{\frac2{m-1}}t\big)^{-\frac1{m-1}}(U^2+V^2)^{\frac m{m-1}}.
\]
And it is easy to see that $H_\vr$ is decreasing along the flow, so that
\[
H_\vr(t,U_\vr(t),V_\vr(t))\leq H_\vr(0,1,1)<\frac{m-2}{2m}2^{\frac{m}{m-1}}.
\]
This implies that 
\begin{\equ}\label{bdd}
U_\vr(t)^2+V_\vr(t)^2\leq C_m\cosh\big(\vr^{\frac2{m-1}}t\big)
\end{\equ}
for some constant $C_m>0$ independent of $\vr$.

Fix $T>0$ and consider $t\in[0,T]$, we have
\begin{\equ}\label{X4}
\aligned
&|U_\vr(t)-U_0(t)|+|V_\vr(t)-V_0(t)|\\
 &\qquad \leq \int_0^t\cosh\big(\vr^{\frac2{m-1}}t\big)^{-\frac1{m-1}}\Big| (U_\vr^2+V_\vr^2)^{\frac1{m-1}}V_\vr- (U_0^2+V_0^2)^{\frac1{m-1}}V_0\Big|ds \\
 &\qquad \quad + \int_0^t\cosh\big(\vr^{\frac2{m-1}}t\big)^{-\frac1{m-1}}\Big| (U_\vr^2+V_\vr^2)^{\frac1{m-1}}U_\vr- (U_0^2+V_0^2)^{\frac1{m-1}}U_0\Big|ds \\
 &\qquad \quad + \int_0^t\Big( 1-\cosh\big(\vr^{\frac2{m-1}}t\big)^{-\frac1{m-1}} \Big)(U_0^2+V_0^2)^{\frac1{m-1}}\big( |U_0|+|V_0| \big) ds \\
 &\qquad \quad 
 + C_m\vr^{\frac2{m-1}}\cosh\big(\vr^{\frac2{m-1}}T\big)^{\frac12}.
\endaligned
\end{\equ}
Since the first two integrands in the right-hand-side of  \eqref{X4} are locally Lipschitz, by \eqref{bdd} and the boundedness of $U_0$ and $V_0$, we have
\[
|U_\vr(t)-U_0(t)|+|V_\vr(t)-V_0(t)|\lesssim \int_0^t \big( |U_\vr-U_0|+|V_\vr-V_0| \big) ds +\vr^{\frac2{m-1}}\cosh\big(\vr^{\frac2{m-1}}T\big)^{\frac12}.
\]
Now, using the Gronwall inequality, we have
\[
|U_\vr(t)-U_0(t)|+|V_\vr(t)-V_0(t)|\lesssim \vr^{\frac2{m-1}}
\]
for $t\in[0,T]$, proving the lemma.
\end{proof}

\begin{proof}[Proof of Proposition \ref{ODE2-proposition1}]
	Suppose  the contrary, that $A_k\cup I_k$ is unbounded for some $k$. Then we can find a sequence $\mu_j\in A_k\cup I_k$ such that $\mu_j\to+\infty$ as $j\to+\infty$. 
	
	By taking $\vr_j=\mu_j^{-1}$, Lemma \ref{ODE2-lemma9} implies that $V_{\vr_j}\to V_0$ uniformly on $[0,T]$ as $j\to\infty$, for any fixed $T>0$. Notice that the solution $(U_0,V_0)$ of Eq. \eqref{ODE2-rescaled-limit} can be explicitly formulated:
	\[
	U_0(t)=\sqrt2\sin\Big( 2^{\frac1{m-1}}t+\frac\pi4 \Big) \quad \text{and} \quad 
	V_0(t)=\sqrt2\cos\Big( 2^{\frac1{m-1}}t+\frac\pi4 \Big).
	\]
	We can take $T>0$ large enough so that $V_0$ changes sign $k+1$ times on $[0,T]$. Then, by Lemma \ref{ODE2-lemma9}, we have $V_{\vr_j}$ changes $k+1$ times on $[0,T]$ for all large $j$. However, due to $\mu_j\in A_k\cup I_k$ and $V_{\vr_j}(t)=\vr_j v_{\mu_j}\big(\vr_j^{2/{(m-1)}}t\big)$, we have $V_{\vr_j}$ should change sign only $k$ times on $(0,+\infty)$. And thus, we get a contradiction.
\end{proof}

\begin{proof}[Proof of Theorem \ref{main thm ODE2}]
Let $\mu_0=\sup A_0$. By Lemma \ref{ODE2-lemma8}, we have $\mu_0\in I_0$. Let now $\nu_0=\sup I_0$. Applying Proposition  \ref{ODE2-proposition1} and Lemma \ref{ODE2-lemma8}, we have $(\nu_0,\nu_0+\vr_0)\subset A_1$ for some $\vr_0>0$. Thus $A_1\neq\emptyset$. Let $\mu_1=\sup A_1$. We have $\mu_1>\nu_0\geq\mu_0$; and so, by Lemma \ref{ODE2-lemma8}, $\mu_1\in I_1$, and then $\nu_1=\sup I_1\in I_1$ and $(\nu_1,\nu_1+\vr_1)\subset A_2$, for some $\vr_1>0$. Iterating this argument, we construct two increasing sequences $\{\mu_j\}$ and $\{\nu_j\}$, $\nu_{j+1}\geq\mu_{j+1}>\nu_j\geq\mu_j$, with $\mu_j\in I_j$ and $(\nu_j,\nu_j+\vr_j)\subset A_{j+1}$, for some $\{\vr_j\}\subset(0,+\infty)$.

Next, we will show that $\mu_j\to+\infty$ as $j\to+\infty$. Suppose, by contradiction, that $\mu_j$ is bounded and $\mu_j\to\mu_\infty$. We can see that $H_{\mu_\infty}(t)>0$ for all $t\in\R$. Indeed, if $H_{\mu_\infty}(t_0)\leq0$ for some finite $t_0>0$, it follows that $(u_{\mu_\infty}(t),v_{\mu_\infty}(t))$ will be trapped in one of the connected components of $\{(u,v)\in\R^2:\, H(t,u,v<0)\}$, for all $t>t_0$. Since Lemma \ref{ODE2-lemma2} implies that $v_{\mu_\infty}$ changes sign a finite number of times in $[0,t_0]$, we have $\mu_\infty\in A_{k_0}$ for some $k_0$. This contradicts the definition of $\mu_\infty$ as $A_{k_0}$ is open. Moreover, $v_{\mu_\infty}$ must change sign infinite many times on $(0,+\infty)$.

Using the facts $H_{\mu_\infty}$ is decreasing on $(0,+\infty)$ and bounded from below, we have $H_{\mu_\infty}'\in L^1(0,+\infty)$. In particular,
\begin{\equ}\label{integrable1}
\cosh(\cdot)^{-\frac1{m-1}}(u_{\mu_\infty}^2+v_{\mu_\infty}^2)^{\frac m{m-1}}\in L^1(0,+\infty).
\end{\equ}
Multiplying by $v_{\mu_\infty}$ (resp. $u_{\mu_\infty}$) the equations in \eqref{reduced-SY4}, we have
\[
\left\{
\aligned
v_{\mu_\infty}u_{\mu_\infty}'&=\cosh(t)^{-\frac1{m-1}}(u_{\mu_\infty}^2+v_{\mu_\infty}^2)^{\frac1{m-1}}v_{\mu_\infty}^2-\frac{m-2}2 u_{\mu_\infty}v_{\mu_\infty} ,\\
-u_{\mu_\infty}v_{\mu_\infty}'&=\cosh(t)^{-\frac1{m-1}}(u_{\mu_\infty}^2+v_{\mu_\infty}^2)^{\frac1{m-1}}u_{\mu_\infty}^2-\frac{m-2}2 u_{\mu_\infty}v_{\mu_\infty}.
\endaligned
\right.
\]
This implies
\[
	v_{\mu_\infty}u_{\mu_\infty}'+u_{\mu_\infty}v_{\mu_\infty}' = \cosh(t)^{-\frac1{m-1}}(u_{\mu_\infty}^2+v_{\mu_\infty}^2)^{\frac1{m-1}}(v_{\mu_\infty}^2-u_{\mu_\infty}^2).
\]
Hence we have $(u_{\mu_\infty}v_{\mu_\infty})'\in L^1(0,+\infty)$, which shows that $u_{\mu_\infty}(t)v_{\mu_\infty}(t)\to C_\infty\in\R$ as $t\to\infty$. Since $v_{\mu_\infty}(t)$ changes sign infinitely many times as $t\to\infty$, we have $C_\infty=0$. This, together with \eqref{integrable1}, implies that $H_{\mu_\infty}(t)\to0$ as $t\to+\infty$.

Therefore, one may take $T>0$ sufficiently large such that $H_{\mu_\infty}(T)<C_0$ (where $C_0>0$ is given by Lemma \ref{ODE2-lemma7}), $u_{\mu_\infty}(T)v_{\mu_\infty}(T)>0$ and $v_{\mu_\infty}$ changes sign $k_{T}$ times on $[0,T]$. By Lemma \ref{ODE2-lemma7}, we have $\mu_\infty\in A_{k_T}\cup I_{k_T}\cup A_{k_T+1}$, reaching another contradiction. 

Finally, in order to see that $\liminf_{t\to+\infty}|u_\mu(t)|+|v_\mu(t)|=+\infty$ for $\mu\in A_k$, let us consider two possibilities: $H_\mu(t)\to-\infty$ and $H_\mu(t)\to H_\infty\in(-\infty,0)$. In the first case, we must have that $u_\mu(t)v_\mu(t)\to+\infty$ as $t\to+\infty$, which directly implies the assertion. In the latter case, we deduce that $u_\mu(t)v_\mu(t)\to C>0$ as $t\to+\infty$. And hence $\cosh(\cdot)^{-\frac1{m-1}}(u_\mu^2+v_\mu^2)^{\frac m{m-1}}$ converges to a positive constant. This shows that $|u_\mu(t)|+|v_\mu(t)|$ grows as $\cosh(t)^{1/{2m}}$ for $t$ large. 

The upper bound of $(u_\mu,v_\mu)$, $\mu\in A_k$ follows from Lemma \ref{ODE2-lemma-unbddness}, and the exponential decay of $(u_\mu,v_\mu)$, $\mu\in I_k$, follows from Corollary \ref{ODE2-lemma-exponential-decay}. Thus, the proof of Theorem \ref{main thm ODE2} is complete.
\end{proof}

\begin{Rem}
The numerical simulations performed on system \eqref{reduced-SY4} indicate the following. For each $k\in\N\cup\{0\}$, starting from $\mu$ larger than some $\mu_k^*\in A_k$, the solution orbits will make a circle around a particular point (in either the first quadrant or the third quadrant) before  going to infinity. As $\mu$ grows, the circle is becoming larger; and once the circle touches the origin, we will have a homoclinic solution of \eqref{reduced-SY4}, which implies $\mu\in I_k$. The set $I_k$ seems to have only one point, and hence $A_k$ are just open intervals. In particular, we conjecture that $\cup_{k\geq0}I_k$ is simply a countable set of discrete points. This is illustrated in the following Fig.~\ref{Fig1}, where numerical experiments are performed on a $3$-dimensional system. The first row shows the solution orbits $(u_\mu,v_\mu)$ on $\R$ with three different initial datum in $A_0$, and specifically $\mu=0.1$, $0.6$ and $0.7$. The second and third rows show the solutions with initial datum $\mu\in A_1$ and $A_2$, respectively

\begin{figure}[ht]
	\centering
	\scalebox{0.21}[0.21]
	{ \includegraphics{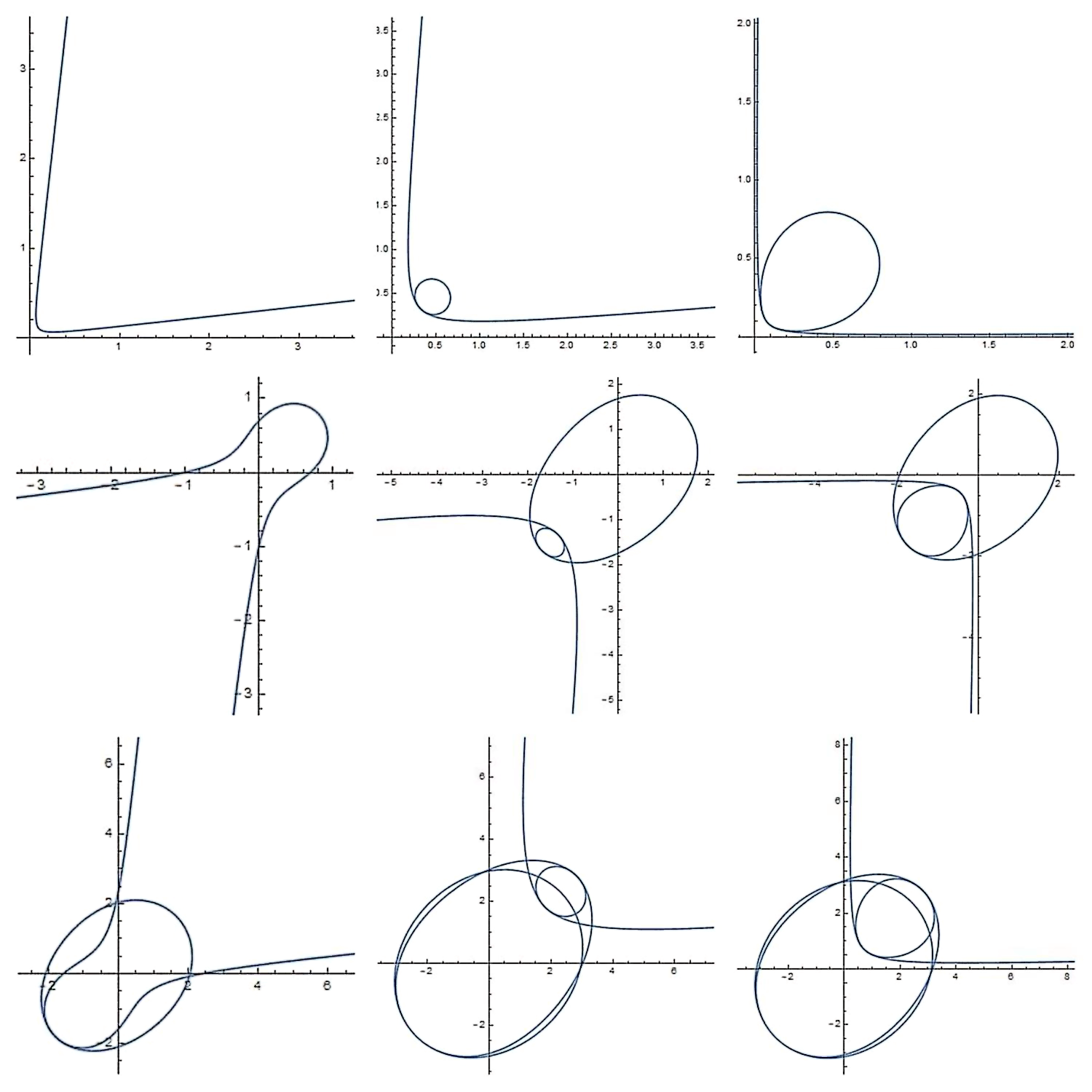}}
	\caption{Unbounded trajectories with initial datum $\mu\in A_k$, $k=0,1,2$.}\label{Fig1}
\end{figure}
\end{Rem}

\section*{Acknowledgements}
Y.S. is partly supported by NSF grant DMS $2154219$, " Regularity {\sl vs} singularity formation in elliptic and parabolic equations".

\newpage

\vspace{2mm}
 {\sc Ali Maalaoui\\
	Department of Mathematics,\\
	Clark University,\\
	Worcester, MA 01610-1477}\\
amaalaoui@clarku.edu

{\sc Yannick Sire\\
Department of Mathematics, Johns Hopkins University,\\
3400 N. Charles Street, Baltimore,
Maryland 21218}\\
ysire1@jhu.edu

\medskip

{\sc Tian Xu\\
 Center for Applied Mathematics, Tianjin University,\\
 300072, Tianjin, China}\\
 xutian@amss.ac.cn

\end{document}